\documentclass[aps,pre,floatfix,twocolumn,showpacs,preprintnumbers,superscriptaddress,amsmath,amssymb,10pt]{revtex4-1}

\usepackage{graphicx,color}
\usepackage[caption=false]{subfig}
\usepackage{amssymb,amsmath,amsthm,bm,mathtools}
\usepackage{enumitem}
\usepackage{multirow}
\usepackage[export]{adjustbox}
\setlist[enumerate]{itemsep=0mm}

\usepackage{ulem}

\newcommand{\ct}{\mathsf{H}}
\newcommand{\tr}{\mathsf{T}}
\newcommand{\etal}{{\em et~al.~}}

\theoremstyle{plain}
\newtheorem{theorem}{Theorem}
\newtheorem{corollary}{Corollary}
\newtheorem{lemma}{Lemma}
\newtheorem{proposition}{Proposition}
\theoremstyle{definition}

\newtheorem{algorithm}{Algorithm}
\newtheorem{assumption}{Assumption}

\theoremstyle{remark}
\newtheorem{remark}{Remark}

\begin{document}

\title{Subspace dynamic mode decomposition for stochastic Koopman analysis}%

\begin{abstract}
The analysis of nonlinear dynamical systems based on the Koopman operator is attracting attention in various applications. Dynamic mode decomposition (DMD) is a data-driven algorithm for Koopman spectral analysis, and several variants with a wide range of applications have been proposed. However, popular implementations of DMD suffer from observation noise on random dynamical systems and generate inaccurate estimation of the spectra of the stochastic Koopman operator. In this paper, we propose {\em subspace DMD} as an algorithm for the Koopman analysis of random dynamical systems with observation noise. Subspace DMD first computes the orthogonal projection of future snapshots to the space of past snapshots and then estimates the spectra of a linear model, and its output converges to the spectra of the stochastic Koopman operator under standard assumptions. We investigate the empirical performance of subspace DMD with several dynamical systems and show its utility for the Koopman analysis of random dynamical systems.
\end{abstract}

\author{Naoya Takeishi}\email[]{takeishi@ailab.t.u-tokyo.ac.jp}
\affiliation{Department of Aeronautics and Astronautics, The University of Tokyo, Bunkyo, Tokyo, Japan}
\author{Yoshinobu Kawahara}
\affiliation{The Institute of Scientific and Industrial Research, Osaka University, Ibaraki, Osaka, Japan}
\affiliation{RIKEN Center for Advanced Intelligence Project, Chuo, Tokyo, Japan}
\author{Takehisa Yairi}
\affiliation{Department of Aeronautics and Astronautics, The University of Tokyo, Bunkyo, Tokyo, Japan}
\maketitle

\section{Introduction}
\label{sec:introduction}

Operator-theoretic approaches for the analysis of dynamical systems, which rely on the Perron--Frobenius operator \cite{Lasota94} or its adjoint, the Koopman operator \cite{Koopman31}, are attracting attention for use in mathematical and engineering applications. The Koopman operator is an infinite-dimensional linear operator that acts on a space of observation functions (observables), and the analysis based on it has been intensively studied recently (see, e.g., \cite{Mezic05,Budisic12,Mezic13}). Several methods for conducting the Koopman spectral analysis have been proposed, such as the generalized Laplace analysis \cite{Budisic12} and the Ulam--Galerkin method \cite{Froyland13}.

Dynamic mode decomposition (DMD) \cite{Rowley09,Schmid10} is a data-driven method that can be utilized for Koopman spectral analysis. It has been applied to a wide range of scientific and engineering subjects including fluid mechanics \cite{Schmid11}, power system analysis \cite{Susuki14}, medical care \cite{Bourantas14}, epidemiology \cite{Proctor15}, robotic control \cite{Berger15}, neuroscience \cite{Brunton16e}, image processing \cite{Kutz16}, nonlinear system identification \cite{Mauroy16}, finance \cite{Mann16}, and chaotic systems \cite{Brunton17}. DMD computes a set of modes along with the corresponding frequencies and decay rates, given a sequence of measurements from the target dynamics. Those modes coincide with the ones obtained by the Koopman spectral analysis under certain conditions, which we briefly review in Section~\ref{sec:background}.

In practice, popular implementations of DMD (e.g., \cite{Schmid10,Tu14}) suffer from {\em observation noise}. Several researchers have addressed this issue; Duke~\etal~\cite{Duke12} and Pan~\etal~\cite{Pan15} conducted error analyses on the DMD algorithms, and Dawson~\etal~\cite{Dawson16} and Hemati~\etal~\cite{Hemati17} proposed reformulating DMD as a total-least-squares problem to treat the observation noise explicitly. Moreover, there is a line of research on the low-rank approximation of dynamics, including optimized DMD \cite{Chen12}, optimal mode decomposition \cite{Wynn13}, sparsity-promoting DMD \cite{Jovanovic14}, and the closed-form solution for a low-rank constrained problem \cite{Heas17}. In addition, Takeishi~\etal~\cite{Takeishi17} suggested a Bayesian formulation of DMD to incorporate uncertainties. Those studies provide clear perspectives on the treatment of the observation noise. Note that, however, they focus on deterministic dynamical systems, i.e., they do not explicitly deal with {\em process noise}, which limits their applicability to situations where the underlying dynamics contain random effects.

In fact, the Koopman analysis can also be applied to dynamical systems with process noise via the {\em stochastic Koopman operator} \cite{Mezic05}. The spectra of the stochastic Koopman operator may convey information on the process noise; Bagheri~\etal~\cite{Bagheri14} investigated the effects of weak noise on the spectra of the Koopman operator for oscillating flows. The DMD algorithms are applicable even to stochastic systems \cite{Williams15a}, {\em unless observation noise is present}. However, the existing variants of DMD do not explicitly consider {\em both} observation and process noise, and, in fact, most of them cannot compute the spectra of the stochastic Koopman operator accurately from noisy observations, which is partly demonstrated in Section~\ref{sec:example} using numerical examples.

In this paper, we present an algorithm based on the stochastic Koopman operator for decomposing nonlinear random dynamical systems from noisy observations. The proposed algorithm is referred to as {\em subspace DMD} because it has a strong connection to the subspace system identification methods developed in control theory. Subspace DMD is aware of both the observation noise and process noise at the same time, and we show its validity with numerical examples.

The remainder of this paper is organized as follows. In Section~\ref{sec:background}, we introduce the fundamental concepts required to understand the purpose and procedures of Koopman analysis and DMD. Section~\ref{sec:noisy} includes the main results of this paper, the algorithm of subspace DMD. In Section~\ref{sec:example}, we introduce numerical examples to show the empirical performance of subspace DMD. In Section~\ref{sec:discussion}, we mention the important elements of DMD that are not fully addressed in this paper. This paper ends with the conclusions in Section~\ref{sec:conclusion}.

\section{Background}
\label{sec:background}

We briefly review the method for decomposing nonlinear dynamical systems based on the spectra and the invariant subspace of the Koopman operator, along with a corresponding numerical procedure. Regarding the theory of Koopman spectral analysis, readers can consult studies such as \cite{Mezic05,Budisic12,Mezic13} for more details.

\subsection{Koopman spectral analysis on Koopman invariant subspace}\label{subsec:koopman}

Consider a discrete-time dynamical system
\begin{equation}
\bm{x}_{t+1} = \bm{f} (\bm{x}_t),\quad
\bm{x} \in \mathcal{M}
\label{eq:ddyn}
\end{equation}
with a map $\bm{f}:\mathcal{M}\to\mathcal{M}$ and time index $t \in \mathbb{T}=\{0\}\cup\mathbb{N}$, where $(\mathcal{M},\Sigma_\mathcal{M},\mu_\mathcal{M})$ is a probability space associated with a phase space $\mathcal{M}$.
Instead of trajectories in the phase space, we analyze the evolution of an {\em observable} $g: \mathcal{M} \to \mathbb{C}$ in a function space $\mathcal{G} \subset L^2(\mathcal{M},\mu_\mathcal{M})$.
{\em Koopman operator} $\mathcal{K}:\mathcal{G}\to\mathcal{G}$ is an infinite-dimensional linear operator defined as
\begin{equation}
\mathcal{K} g (\bm{x}) \coloneqq g (\bm{f}(\bm{x})) .
\end{equation}

Here, suppose that there exists a subspace $G \subset \mathcal{G}$ that is invariant to $\mathcal{K}$, i.e., $\exists G~\text{s.t.}~\mathcal{K}g\in G, \forall g\in G$ \footnote{The theory of decomposition based on the spectra of the Koopman operator does not necessarily require the existence of the invariant subspace. We introduce the notion of the invariant subspace here merely for clarifying the connection between the Koopman spectral analysis and dynamic mode decomposition.}.
Let us consider the restriction of $\mathcal{K}$ to $G$ and denote it by $K$. If $G$ is finite-dimensional, then $K$ also becomes a finite-dimensional linear operator.
Suppose we have a set of observables $\{g_1,\dots,g_n\}$ ($n<\infty$) that spans $G$, and let $\bm{K}$ be the representation of $K$ with regard to $\{g_1,\dots,g_n\}$, i.e.,
\begin{equation}\label{eq:k_matrix}
\begin{bmatrix} K g_1 & \cdots & K g_n \end{bmatrix}^\tr = \bm{K} \bm{g},
\end{equation}
where $\bm{g} = \begin{bmatrix} g_1 & \cdots & g_n \end{bmatrix}^\tr$.
Now let $\varphi$ be the eigenfunction of $K$ corresponding to an eigenvalue $\lambda$. Then, $\varphi$ with regard to $\{g_1,\dots,g_n\}$ is expressed as
\begin{equation}
\varphi(\bm{x}) = \bm{z}^\ct \bm{g} (\bm{x}),
\end{equation}
where $\bm{z}$ is the left-eigenvector of $\bm{K}$ corresponding to eigenvalue $\lambda$, since
\begin{equation*}
K \left( \bm{z}^\ct \bm{g} (\bm{x}) \right)
= \bm{z}^\ct \bm{K} \bm{g} (\bm{x})
= \lambda \bm{z}^\ct \bm{g} (\bm{x}).
\end{equation*}

Let $\bm{w}_i$ and $\bm{z}_i$ respectively be the right- and the left-eigenvector of $\bm{K}$ corresponding to an eigenvalue $\lambda_i$ for $i=1,\dots,n$.
In the sequel, without loss of generality, we assume that $\bm{w}$ and $\bm{z}$ are normalized so that $\bm{w}_{i'}^\ct\bm{z}_{i}=\delta_{i'i}$ ($\delta_{i'i}$ is $0$ if $i'=i$ and $1$ otherwise).
We assume that all the eigenvalues of $K$ are distinct, i.e., their multiplicities are one.
Then, any values of $\bm{g}$ are expressed as
\begin{equation}\label{eq:kmd_r1}
\bm{g}(\bm{x})
= \sum_{i=1}^n \bm{z}_i^\ct \bm{g}(\bm{x}) \bm{w}_i
= \sum_{i=1}^n \varphi_i(\bm{x}) \bm{w}_i,
\end{equation}
where $\varphi_i$ is the eigenfunction of $K$ corresponding to eigenvalue $\lambda_i$.
Applying $K$ on both sides of Eq.~\eqref{eq:kmd_r1} repeatedly starting at $\bm{x}=\bm{x}_0$, we obtain the modal decomposition of the values of the observables, i.e.,
\begin{equation}
\bm{g}(\bm{x}_t) = \sum_{i=1}^n \lambda_i^t \bm{c}_i,\quad
\bm{c}_i = \varphi_i(\bm{x}_0)\bm{w}_i,
\end{equation}
where coefficient $\bm{c}_i$ (or $\bm{w}_i$) is referred to as a {\em Koopman mode}.

The same sort of discussion is also possible for a continuous-time dynamical system
\begin{equation}
\frac{\mathrm{d}\bm{x}}{\mathrm{d}t} = \bm{f}(\bm{x}),\quad
\bm{x} \in \mathcal{M}.
\end{equation}
Instead of the Koopman operator, the {\em Koopman semigroup} $\{\mathcal{K}_c^t\}_{t\in\mathbb{R}^+}$ on this dynamical system is defined as
\begin{equation}
\mathcal{K}_c^t g(\bm{x}) \coloneqq g(\bm\phi(\bm{x},t)),
\end{equation}
where $\bm\phi(\bm{x},t)$ is the flow map that takes $\bm{x}$ as the initial state and returns the state after a time interval of length $t$. The infinitesimal generator of the Koopman semigroup is given as
\begin{equation}
\mathcal{K}_c = \lim_{t \to 0} \frac{\mathcal{K}_c^t g - g}{t}.
\end{equation}
Let $\lambda_c$ be an eigenvalue of $\mathcal{K}_c$. This ``continuous-time'' eigenvalue can be computed from the corresponding ``discrete-time'' eigenvalue $\lambda$ by $\lambda_c = \ln(\lambda)/\Delta t$, where $\Delta t$ is the temporal interval in discrete-time dynamical system~\eqref{eq:ddyn}.

\subsection{Dynamic mode decomposition}\label{subsec:dmd}

DMD \cite{Rowley09,Schmid10,DMDBook} is a decomposition method for numerical datasets whose output converges to the modal decomposition via the Koopman operator under some conditions.
Suppose that $\{g_1,\dots,g_n\}$ spans $G$ and that we have data matrices generated with $\bm{g}=\begin{bmatrix}g_1&\cdots&g_n\end{bmatrix}^\tr$:
\begin{equation}\label{eq:datapair}
\begin{aligned}
\bm{Y}_0 &= \begin{bmatrix} \bm{g}(\bm{x}_0) & \cdots & \bm{g}(\bm{x}_{m-1}) \end{bmatrix}\in\mathbb{C}^{n \times m} \quad\text{and} \\
\bm{Y}_1 &= \begin{bmatrix} \bm{g}(\bm{x}_1) & \cdots & \bm{g}(\bm{x}_m)) \end{bmatrix}\in\mathbb{C}^{n \times m}.
\end{aligned}
\end{equation}
The popular algorithm of DMD \cite{Schmid10,Tu14} leverages a compact singular value decomposition (SVD) to avoid a direct eigendecomposition of a large matrix, and its procedure is shown in Algorithm~\ref{alg:dmd}.
\begin{algorithm}[DMD \cite{Schmid10,Tu14}]\label{alg:dmd}\leavevmode
\begin{enumerate}
\item Build a pair of data matrices $(\bm{Y}_0,\bm{Y}_1)$ as in Eq.~\eqref{eq:datapair}.
\item Compute the compact SVD as $\bm{Y}_0 = \bm{U}_r \bm{S}_r \bm{V}_r^\ct$ with $\bm{U}_r \in \mathbb{C}^{n \times r}$, $\bm{S}_r \in \mathbb{C}^{r \times r}$ and $\bm{V}_r \in \mathbb{C}^{m \times r}$, where $r=\operatorname{rank}(\bm{Y}_0)$. \label{alg:dmd:svd}
\item Define matrix $\tilde{\bm{A}} = \bm{U}_r^\ct \bm{Y}_1 \bm{V}_r \bm{S}_r^{-1}$.
\item Compute the eigenvalues $\lambda$ and eigenvectors $\tilde{\bm{w}}$ of $\tilde{\bm{A}}$.
\item Return {\em dynamic modes} $\bm{w} = \lambda^{-1} \bm{Y}_1 \bm{V}_r \bm{S}_r^{-1} \tilde{\bm{w}}$ and the corresponding eigenvalues $\lambda$.
\end{enumerate}
\end{algorithm}

The convergence of an output of Algorithm~\ref{alg:dmd} in the large sample limit can be shown with the assumption of ergodicity as follows.
\begin{assumption}\label{asmp:ergodicity1}
The time average of a measurable function $\phi:\mathcal{M}\to\mathbb{C}$ converges to its space average, i.e.,
\begin{equation*}
\lim_{m\to\infty} \frac1m \sum_{j=0}^{m-1} \phi (\bm{x}_j) = \mathbb{E}_\mathcal{M} [\phi(\bm{x})] = \int_\mathcal{M} \phi(\bm{x}) d\mu_\mathcal{M},
\end{equation*}
for almost all $\bm{x}_0\in\mathcal{M}$.
\end{assumption}
\begin{proposition}\label{prop:dmd}
Suppose Assumption~\ref{asmp:ergodicity1} holds. If all the modes are sufficiently excited in the data (i.e., $r=\dim(G)$) and all the nonzero eigenvalues of $\bm{A}=\bm{Y}_1\bm{Y}_0^\dagger$ are distinct, then the dynamic modes calculated by Algorithm~\ref{alg:dmd} converge to the eigenvectors corresponding to the non-zero eigenvalues of $\bm{K}$ in $m\to\infty$ with probability one.
\end{proposition}
\begin{proof}
Taking the inner product of both sides of Eq.~\eqref{eq:k_matrix} with $\bm{g}$, we have
\begin{equation*}
\bm{K} \bm{G}_0 = \bm{G}_1,~
\bm{G}_0 = \mathbb{E}_\mathcal{M} \left[ \bm{g}\bm{g}^\ct \right],~
\bm{G}_1 = \mathbb{E}_\mathcal{M} \left[  (\bm{g} \circ \bm{f}) \bm{g}^\ct \right],
\end{equation*}
and thus the minimum-norm solution for $\bm{K}$ is given as $\bm{K} = \bm{G}_1 \bm{G}_0^\dagger$, where $\bm{G}_0^\dagger$ is the Moore--Penrose pseudoinverse of $\bm{G}_0$.
In contrast, from the definition of $\bm{A}$,
\begin{equation*}
\bm{A} = \hat{\bm{G}}_1 \hat{\bm{G}}_0^\dagger, \quad
\hat{\bm{G}}_0 = \frac1m \bm{Y}_0 \bm{Y}_0^\ct, \quad
\hat{\bm{G}}_1 = \frac1m \bm{Y}_1 \bm{Y}_0^\ct,
\end{equation*}
and by Assumption~\ref{asmp:ergodicity1}, empirical matrices $\hat{\bm{G}}_0$ and $\hat{\bm{G}}_1$ converge to $\bm{G}_0$ and $\bm{G}_1$, respectively, in $m \to \infty$ with probability one. Moreover, because $r=\dim(G)$, the rank of $\hat{\bm{G}}_0$ is always $\dim(G)$ and thus $\hat{\bm{G}}_0^\dagger$ converges to $\bm{\bm{G}}_0^\dagger$ \cite{Rakocevic97}. Further, because Algorithm~\ref{alg:dmd} returns the eigenvectors corresponding to all the non-zero eigenvalues of $\bm{A}$ \cite{Tu14}, if those non-zero eigenvalues are distinct (i.e., their multiplicity is one), the outputs of Algorithm~\ref{alg:dmd} are continuous with respect to $\bm{A}$. Therefore, the dynamic modes converge to the eigenvectors corresponding to the non-zero eigenvalues of $\bm{K}$ with probability one.
\end{proof}
\begin{remark}
We have shown the convergence utilizing the assumption of ergodicity. However, the algorithm defined by Tu~\etal \cite{Tu14} does not require the sequential sampling in Eq.~\eqref{eq:datapair} as long as the corresponding columns of the data matrices are sampled with a fixed temporal interval. One can also prove the convergence (in probability) for this case using the law of large numbers if one assumes the snapshots in $\bm{Y}_0$ are independently sampled from $\mu_\mathcal{M}$.
\end{remark}
\begin{remark}
For an output of DMD to coincide with the modal decompositions via the Koopman operator, data must be generated from observables that span a subspace invariant to the Koopman operator. In this paper, we assume that the data in hand are intrinsically generated with such observables $\bm{g}$, as most of the existing DMD variants do. Some ways to design such observables manually are reviewed in Section~\ref{sec:discussion}.
\end{remark}
%

\section{Stochastic Koopman analysis with noisy observations}\label{sec:noisy}

In the previous section, we considered systems with no stochastic elements. However, stochasticity often comprises an essential part of a variety of physical phenomena and sensing. In this section, we introduce the notions of process noise on dynamics and observation noise on observables, and discuss Koopman analysis and DMD for stochastic noisy systems.

\subsection{Process noise on dynamics}

Instead of deterministic dynamical system~\eqref{eq:ddyn}, consider a random dynamical system (RDS) \cite{Arnold98}
\begin{equation}\label{eq:rdyn}
\begin{aligned}
\bm{x}_{t+1} = \bm{f}_\Omega (\bm{x}_t, \omega_t),\quad
\bm{x} \in \mathcal{M},\quad
\omega \in \Omega
\end{aligned}
\end{equation}
with a measure-preserving base flow $\vartheta:\Omega\to\Omega$, where $(\Omega,\Sigma_\Omega,\mu_\Omega)$ is a probability space of process noise. We assume that $\omega_t$ is independent from $\bm{x}_0,\dots,\bm{x}_t$.
A one-step evolution of observables $g$ with regard to the RDS can be characterized by {\em stochastic Koopman operator} $\mathcal{K}_\Omega$ \cite{Mezic05}, defined as
\begin{equation}
\mathcal{K}_\Omega g (\bm{x}) \coloneqq \mathbb{E}_\Omega \left[ g(\bm{f}_\Omega(\bm{x},\omega)) \right],
\end{equation}
where $\mathbb{E}_\Omega[\cdot]$ denotes expectation in sample space $\Omega$.
Note that deterministic Koopman operator $\mathcal{K}$ can be regarded as a special case of $\mathcal{K}_\Omega$.
Now let $K_\Omega$ be the restriction of $\mathcal{K}_\Omega$ to its invariant subspace $G$, suppose that a set of observables $\{g_1,\dots,g_n\}$ spans $G$, and let $\bm{g}=\begin{bmatrix}g_1 & \dots & g_n\end{bmatrix}^\tr$.
In addition, let $\bm{K}_\Omega \in \mathbb{C}^{n \times n}$ be the representation of $K_\Omega$ with regard to the components of $\bm{g}$. Then, we have
\begin{equation}\label{eq:g_def}
\bm{g}(\bm{x}_{t+1}) = \bm{K}_\Omega\bm{g}(\bm{x}_t) + \bm{e}_t,
\end{equation}
where
\begin{equation}
\bm{e}_t \coloneqq \bm{g}(\bm{f}(\bm{x}_t,\omega_t)) - \mathbb{E}_\Omega \left[ \bm{g}(\bm{f}(\bm{x}_t,\omega_t)) \right].
\end{equation}
Given $\bm{x}_0$, the solution of \eqref{eq:g_def} then becomes
\begin{equation}\label{eq:g_sol}
\bm{g}(\bm{x}_t) = \bm{K}_\Omega^t \bm{g}(\bm{x}_0) + \sum_{k=0}^{t-1} \bm{K}_\Omega^{t-k-1} \bm{e}_k.
\end{equation}
The modal decomposition of $\bm{g}$ via $K_\Omega$ can be obtained likewise, as shown in Section~\ref{subsec:koopman}. Regarding the characteristics of the spectra of the stochastic Koopman operator, Bagheri~\cite{Bagheri14} elaborated on the effects of weak noise in a limit cycle, and Williams~\etal~\cite{Williams15a} applied a variant of DMD to the data obtained from a stochastic differential equation.

The standard DMD (Algorithm~\ref{alg:dmd}) is also applicable to the RDS and $\mathcal{K}_\Omega$ if there is no observation noise and the ergodicity (Assumption~\ref{asmp:ergodicity1}) also holds for the RDS. This can be shown in a manner similar to the one in Proposition~\ref{prop:dmd}, except for the definition of $\bm{G}_0$ and $\bm{G}_1$, as follows. Let $\bm{Y}_0$ and $\bm{Y}_1$ be the data matrices generated from RDS $\bm{f}_\Omega$ and observable $\bm{g}$ as in Eq.~\eqref{eq:datapair}, and let us assume the whiteness on process noise.
\begin{assumption}\label{asmp:e_cov}
Process noise $\omega$ is independently and identically distributed in time, i.e., for all $t',t\in\mathbb{T}$,
\begin{equation*}
\mathbb{E}_\Omega \left[ \bm{e}_{t'}\bm{e}_{t}^\ct \right] = \bm{P}\delta_{t't}
\end{equation*}
for some $\bm{P}\in\mathbb{C}^{n \times n}$.
\end{assumption}
Then, from the law of large numbers and the assumption of ergodicity, the empirical matrices
\begin{equation*}
\begin{aligned}
\hat{\bm{G}}_0 &= \frac1m \bm{Y}_0\bm{Y}_0 = \frac1m\sum_{j=0}^{m-1}\bm{g}(\bm{x}_j)\bm{g}(\bm{x}_j)^\ct \quad \text{and} \\
\hat{\bm{G}}_1 &= \frac1m \bm{Y}_1\bm{Y}_0 = \frac1m\sum_{j=0}^{m-1}\bm{g}(\bm{f}_\Omega(\bm{x}_j,\omega_j)) \bm{g}(\bm{x}_j)^\ct
\end{aligned}
\end{equation*}
respectively converge to
\begin{equation*}
\begin{aligned}
\bm{G}_0 &= \mathbb{E}_\mathcal{M} \left[ \bm{g}(\bm{x})\bm{g}(\bm{x})^\ct \right]
\quad \text{and} \\
\bm{G}_1 &= \mathbb{E}_\mathcal{M} \left[ \mathbb{E}_\Omega \left[ \bm{g}(\bm{f}_\Omega(\bm{x},\omega)) \right] \bm{g}(\bm{x})^\ct \right] \\
&= \int_{\mathcal{M} \times \Omega} \bm{g}(\bm{f}_\Omega(\bm{x},\omega)) \bm{g}(\bm{x})^\ct d\mu_\mathcal{M} d\mu_\Omega
\end{aligned}
\end{equation*}
with probability one. One can use this convergence property to show the applicability of Algorithm~\ref{alg:dmd} for the RDS, as in the proof of Proposition~\ref{prop:dmd}.

\subsection{Observation noise on observables}

In addition to the process noise, let us take the observation noise into account. Consider a new (noisy) observable $\bm{h}: \mathcal{M} \times S \to \mathbb{C}^n$:
\begin{equation}\label{eq:h_def}
\bm{h}(\bm{x}_t,s_t) \coloneqq \bm{g}(\bm{x}_t) + \bm{w}(s_t),\quad
\bm{x} \in \mathcal{M},\quad
s \in S,
\end{equation}
where $\bm{w}: S \to \mathbb{C}^n$ is a random variable on a probability space $(S,\Sigma_S,\mu_S)$ of the observation noise. Hereafter, we denote $\bm{w}(s_t)$ by $\bm{w}_t$ for notational simplicity. Now assume that $s$ is independent from $\bm{x}$ and that $\bm{w}$ is a white noise.
\begin{assumption}\label{asmp:w_cov}
Observation noise $\bm{w}$ is zero-mean, has time-invariant finite variance, and is temporally uncorrelated, i.e., for all $t',t\in\mathbb{T}$,
\begin{equation*}
\begin{gathered}
\mathbb{E}_S \left[ \bm{w}_t \right] = 0,\quad
\mathbb{E}_S \left[ \bm{w}_{t'} \bm{w}_{t}^\ct \right] = \bm{Q}\delta_{t't},\\
\mathbb{E}_{\Omega,S} \left[ \bm{e}_{t'} \bm{w}_{t}^\ct \right] = \bm{R}\delta_{t't},
\end{gathered}
\end{equation*}
for some $\bm{Q},\bm{R}\in\mathbb{C}^{n \times n}$.
\end{assumption}
Note that under the presence of observation noise, an output of DMD (Algorithm~\ref{alg:dmd}) no longer converges to the spectra of the Koopman operator. An output of total-least-squares DMD \cite{Dawson16,Hemati17} is unbiased even for noisy observations as long as the dynamics are deterministic, but it is biased as a realization of $K_\Omega$ for the RDS. These inconsistencies in the existing methods are partly revealed in the numerical examples in Section~\ref{sec:example}.

\subsection{Statistics of noisy observables on RDS}

We would like to develop a DMD algorithm for stochastic Koopman analysis that is always aware of {\em both} the process noise and observation noise. To this end, we summarize the statistics of noisy observable $\bm{h}$ on RDS $\bm{f}_\Omega$.
Proofs of the lemmas in this section are deferred to Appendix for clarity of presentation.

First, assume that $\bm{g}$ is quasi-stationary (see \cite{Ljung99}), i.e.,
\begin{assumption}\label{asmp:quasi}
For almost all $\bm{x}_0\in\mathcal{M}$ and all $t',t\in\mathbb{T}$,
\begin{equation*}
\begin{aligned}
\mathbb{E}_\Omega \left[ \bm{g}(\bm{x}_{t}) \right] &= \bm{m}_{t}, \quad
\vert \bm{m}_{t} \vert < \infty, \\
\mathbb{E}_\Omega \left[ \bm{g}(\bm{x}_{t'}) \bm{g}(\bm{x}_{t})^\ct \right] &= \bm{G}_{t',t}, \quad \Vert\bm{G}_{t',t}\Vert_F < \infty, \\
\mathbb{E}_\mathcal{M} \left[ \bm{G}_{t,t} \right] &= \bm{G},
\end{aligned}
\end{equation*}
for some $\bm{G}\in\mathbb{C}^{n \times n}$.
\end{assumption}

Then, the second-order moment of $\bm{g}$, $\bm{G}_{t',t}$, satisfies the following properties.
\begin{lemma}\label{lem:p_g}
$\bm{G}_{t',t}$ is expressed as
\begin{equation*}
\bm{G}_{t',t} =
\begin{cases}
\bm{K}_\Omega^{t'-t} \bm{G}_{t,t}, & t' \geq t, \\
\bm{G}_{t',t'} (\bm{K}_\Omega^{t-t'})^\ct, & t' < t.
\end{cases}
\end{equation*}
\end{lemma}
\begin{corollary}
Denote $\mathbb{E}_\mathcal{M}\left[\bm{G}_{t+\tau,t}\right]$ by $\bm{G}_\tau$. Then,
\begin{equation*}
\bm{G}_\tau = \bm{K}_\Omega^\tau \bm{G}.
\end{equation*}
\end{corollary}

Now let us define $\bm{H}_{t',t}\coloneqq\mathbb{E}_{\Omega,S} \left[ \bm{h}(\bm{x}_{t'}) \bm{h}(\bm{x}_{t})^\ct \right]$, where we have dropped argument $s$ of $\bm{h}$ for ease of notation. Then, $\bm{H}_{t',t}$ satisfies the following properties.
\begin{lemma}\label{lem:p_h}
$\bm{H}_{t',t}$ is expressed as
\begin{equation*}
\bm{H}_{t',t} =
\begin{cases}
\bm{K}_\Omega^{t'-t-1} \left( \bm{K}_\Omega \bm{G}_{t,t}  + \bm{R} \right), & t' > t, \\
 \bm{G}_{t,t}  + \bm{Q}, & t' = t, \\
\bm{H}_{t,t'}^\ct, & t' < t.
\end{cases}
\end{equation*}
\end{lemma}
\begin{corollary}
Denote $\mathbb{E}_\mathcal{M} \left[ \bm{H}_{t+\tau,t} \right]$ by $\bm{H}_\tau$. Then,
\begin{equation*}
\bm{H}_\tau
=\begin{cases}
 \bm{K}_\Omega^{\tau-1} \left( \bm{K}_\Omega \bm{G}  + \bm{R} \right), & \tau>0, \\
 \bm{G}  + \bm{Q}, & \tau=0.
\end{cases}
\end{equation*}
\end{corollary}
%

\subsection{Subspace DMD}

Finally, we introduce a numerical method to compute an instance of the stochastic Koopman operator given noisy observations, namely, {\em subspace DMD}.
Analogously to Eq.~\eqref{eq:datapair}, let us define the data matrix as a concatenation of $m$ observations starting at time $t$, i.e.,
\begin{equation}
\bm{Y}_t = \begin{bmatrix} \bm{h}(\bm{x}_t) & \dots & \bm{h}(\bm{x}_{t+m-1}) \end{bmatrix}\in\mathbb{C}^{n \times m}.
\end{equation}
Then, using a data quadruple $(\bm{Y}_0,\bm{Y}_1,\bm{Y}_2,\bm{Y}_3)$, we can obtain a calculation for $K_\Omega$ using the following theorem.
\begin{theorem}\label{thm:subspace}
Define $\bm{Y}_p,\bm{Y}_f\in\mathbb{C}^{2n \times m}$ by
\begin{equation}\label{eq:datafp}
\bm{Y}_p=\begin{bmatrix}\bm{Y}_0^\tr & \bm{Y}_1^\tr\end{bmatrix}^\tr,\quad
\bm{Y}_f=\begin{bmatrix}\bm{Y}_2^\tr & \bm{Y}_3^\tr\end{bmatrix}^\tr,
\end{equation}
and let $\bm{O}=\bm{Y}_f\mathbb{P}_{\bm{Y}_p^\ct}\in\mathbb{C}^{2n \times m}$ be the orthogonal projection of rows of $\bm{Y}_f$ onto the row space of $\bm{Y}_p$.
Further, consider a compact SVD
\begin{equation}\label{eq:thm:svd}
\bm{O}=\bm{U}_q\bm{S}_q\bm{V}_q^\ct
\end{equation}
with $\bm{U}_q\in\mathbb{C}^{2n \times q}$, $\bm{S}_q\in\mathbb{C}^{q \times q}$, and $\bm{V}_q\in\mathbb{C}^{m \times q}$, where $q=\operatorname{rank}(\bm{O})$.
Moreover, let $\bm{U}_{q1}$ be the first $n$ rows and $\bm{U}_{q2}$ be the last $n$ rows of $\bm{U}_q$.
If $\operatorname{rank}(\bm{Y}_p)=2n$ and $\operatorname{rank}(\bm{K}_\Omega\bm{G}+\bm{R})=n$, then in $m\to\infty$,
\begin{equation}\label{eq:mainconv}
\bm{U}_{q2}\bm{U}_{q1}^\dagger \to \bm{K}_\Omega
\end{equation}
with probability one.
\end{theorem}
\begin{proof}
Let $\hat{\bm{H}}$ be the empirical matrix such that
\begin{equation*}
\hat{\bm{H}}_{t+\tau,t} = \frac1m \bm{Y}_{t+\tau} \bm{Y}_t^\ct.
\end{equation*}
In $m\to\infty$, $\hat{\bm{H}}_{t+\tau,t}$ converges to $\bm{H}_\tau$ with probability one for all $t\in\mathbb{T}$ and $\tau \geq 0$, because, from the law of large numbers and the assumption of ergodicity,
\begin{equation*}
\begin{aligned}
\frac1m \bm{Y}_{t+\tau} \bm{Y}_t^\ct
&= \frac1m \sum_{j=t}^{t+m-1} \bm{h}(\bm{x}_{j+\tau})  \bm{h}(\bm{x}_{j})^\ct \\
&\to \int_{\mathcal{M} \times \Omega \times S} \bm{h}(\bm{x}_{t+\tau})\bm{h}(\bm{x}_t)^\ct d\mu_\mathcal{M} d\mu_\Omega d\mu_S \\
&= \bm{H}_\tau.
\end{aligned}
\end{equation*}
Because we have assumed $\operatorname{rank}(\bm{Y}_p)=2n$, in $m\to\infty$,
\begin{equation}\label{eq:o_decomp}
\begin{aligned}
\bm{O}
&= \bm{Y}_f \bm{Y}_p^\ct \left( \bm{Y}_p\bm{Y}_p^\ct \right)^{-1} \bm{Y}_p \\
&= \begin{bmatrix} \hat{\bm{H}}_{2,0} & \hat{\bm{H}}_{2,1} \\ \hat{\bm{H}}_{3,0} & \hat{\bm{H}}_{3,1} \end{bmatrix} \begin{bmatrix} \hat{\bm{H}}_{0,0} & \hat{\bm{H}}_{0,1} \\ \hat{\bm{H}}_{1,0} & \hat{\bm{H}}_{1,1} \end{bmatrix}^{-1} \bm{Y}_p \\
&\to \begin{bmatrix} \bm{H}_2 & \bm{H}_1 \\ \bm{H}_3 & \bm{H}_2 \end{bmatrix} \begin{bmatrix} \bm{H}_0 & \bm{H}_1^\ct \\ \bm{H}_1 & \bm{H}_0 \end{bmatrix}^{-1} \bm{Y}_p \\
&= \begin{bmatrix} \bm{I} \\ \bm{K}_\Omega \end{bmatrix} \begin{bmatrix} \bm{K}_\Omega\bm{D} & \bm{D} \end{bmatrix} \begin{bmatrix} \bm{G}+\bm{Q} & \bm{D}^\ct \\ \bm{D} & \bm{G}+\bm{Q} \end{bmatrix}^{-1} \bm{Y}_p \\
&= \bm{O}_1 \bm{O}_2
\end{aligned}
\end{equation}
with probability one, where $\bm{D}=\bm{K}_\Omega\bm{G}+\bm{R}\in\mathbb{C}^{n \times n}$ and
\begin{equation*}
\begin{aligned}
\bm{O}_1 &= \begin{bmatrix} \bm{I} \\ \bm{K}_\Omega \end{bmatrix},\\
\bm{O}_2 &= \bm{D}^\tr \begin{bmatrix} \bm{K}_\Omega \\ \bm{I} \end{bmatrix}^\tr \begin{bmatrix} \bm{G}+\bm{Q} & \bm{D}^\ct \\ \bm{D} & \bm{G}+\bm{Q} \end{bmatrix}^{-1} \bm{Y}_p.
\end{aligned}
\end{equation*}
Because we have assumed $\operatorname{rank}(\bm{D})=n$, the rank of both $\bm{O}_1$ and $\bm{O}_2$ is $n$. Hence, in $m\to\infty$, $q$ also becomes $n$. Remember that by compact SVD~\eqref{eq:thm:svd}, we have the decomposition of $\bm{O}$ into two rank-$n$ matrices, i.e.,
\begin{equation*}
\bm{O}=\left(\bm{U}_q\bm{S}_q^{1/2}\right)\left(\bm{S}_q^{1/2}\bm{V}_q^\ct\right).
\end{equation*}
Therefore, from Eq.~\eqref{eq:o_decomp}, in $m\to\infty$, we have
\begin{equation*}
\bm{U}_q\bm{S}_q^{1/2} \to \bm{O}_1\bm{T} =
\begin{bmatrix} \bm{T} \\ \bm{K}_\Omega\bm{T} \end{bmatrix}
\end{equation*}
with probability one, where $\bm{T}\in\mathbb{C}^{n \times n}$ is an arbitrary unitary matrix. Consequently, $\bm{U}_{q1}$ and $\bm{U}_{q2}$ become $\bm{T}$ and $\bm{K}_\Omega\bm{T}$ respectively, and Eq.~\eqref{eq:mainconv} holds.
\end{proof}

Based on Theorem~\ref{thm:subspace}, we present a {\em subspace DMD} algorithm as follows.
\begin{algorithm}[Subspace DMD]\label{alg:sdmd}\leavevmode
\begin{enumerate}
\item Build matrices $\bm{Y}_p$ and $\bm{Y}_f$ like Eq.~\eqref{eq:datafp} from a data quadruple $(\bm{Y}_0,\bm{Y}_1,\bm{Y}_2,\bm{Y}_3)$.
\item Compute orthogonal projection $\bm{O}=\bm{Y}_f\mathbb{P}_{\bm{Y}_p^\ct}$.
\item\label{alg:sdmd:svd2} Compute compact SVD $\bm{O}=\bm{U}_q\bm{S}_q\bm{V}_q^\ct$ and define $\bm{U}_{q1}$ and $\bm{U}_{q2}$ by the first and the last $n$ rows of $\bm{U}_q$, respectively.
\item Compute compact SVD $\bm{U}_{q1} = \bm{U}\bm{S}\bm{V}^\ct$ and define $\tilde{\bm{A}}=\bm{U}^\ct\bm{U}_{q2}\bm{V}\bm{S}^{-1}$.
\item Compute the eigenvalues $\lambda$ and eigenvectors $\tilde{\bm{w}}$ of $\tilde{\bm{A}}$.
\item Return dynamic modes $\bm{w}=\lambda^{-1}\bm{U}_{q2}\bm{V}\bm{S}^{-1}\tilde{\bm{w}}$ and corresponding eigenvalues $\lambda$.
\end{enumerate}
\end{algorithm}
\begin{remark}
With subspace DMD, we can naturally conduct a low-rank approximation of dynamics by replacing the compact SVD in Step~\ref{alg:sdmd:svd2} with a truncated SVD. In contrast, in Algorithm~\ref{alg:dmd} and total-least-squares DMD \cite{Dawson16,Hemati17}, the low-rank approximation is achieved via the truncated proper orthogonal decomposition (POD). Note that there is also a line of research on the low-rank approximation of DMD, such as \cite{Chen12,Wynn13,Jovanovic14,Heas17}.
\end{remark}
\begin{remark}
Again note that we suppose that data are obtained with observable $\bm{g}$ that spans a subspace invariant to the Koopman operator, as in Algorithm~\ref{alg:dmd} and other popular variants of DMD. See Section~\ref{sec:discussion} for ways to design such observables.
\end{remark}
%

\section{Numerical Examples}\label{sec:example}

We present numerical examples for the application of subspace DMD to several types of dynamical systems to show its empirical performance. When describing target dynamical systems in the following examples, we denote Gaussian white process noise by $\bm{e}$ and Gaussian white observation noise by $\bm{w}$. The standard deviation of the process noise is referred to as $\sigma_p$ and that of the observation noise as $\sigma_o$. Moreover, we denote the number of snapshots fed into algorithms by $m$ and the dimensionality of the data by $n$.

\begin{figure}[t]
\centering
\subfloat[]{\includegraphics[clip,width=0.49\textwidth]{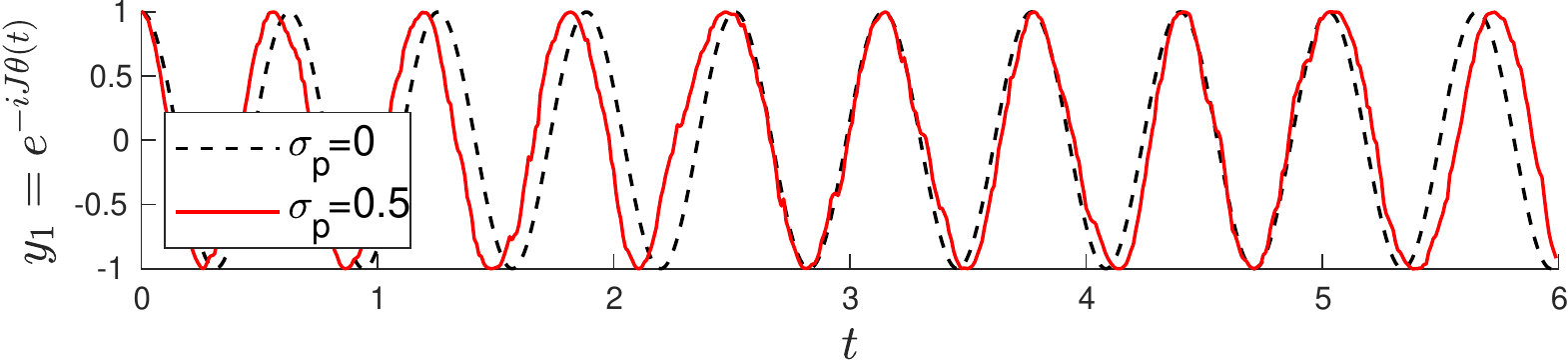}\label{fig:exp_sl_data}}
\\
\subfloat[]{\includegraphics[clip,width=0.235\textwidth]{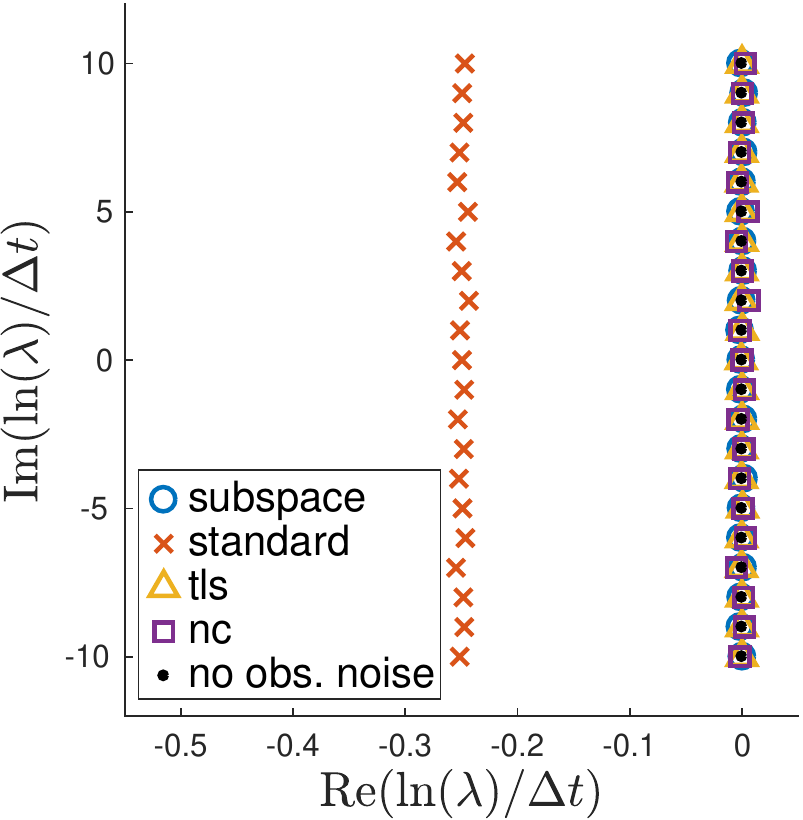}\label{fig:exp_sl_a}}
\hfill
\subfloat[]{\includegraphics[clip,width=0.235\textwidth]{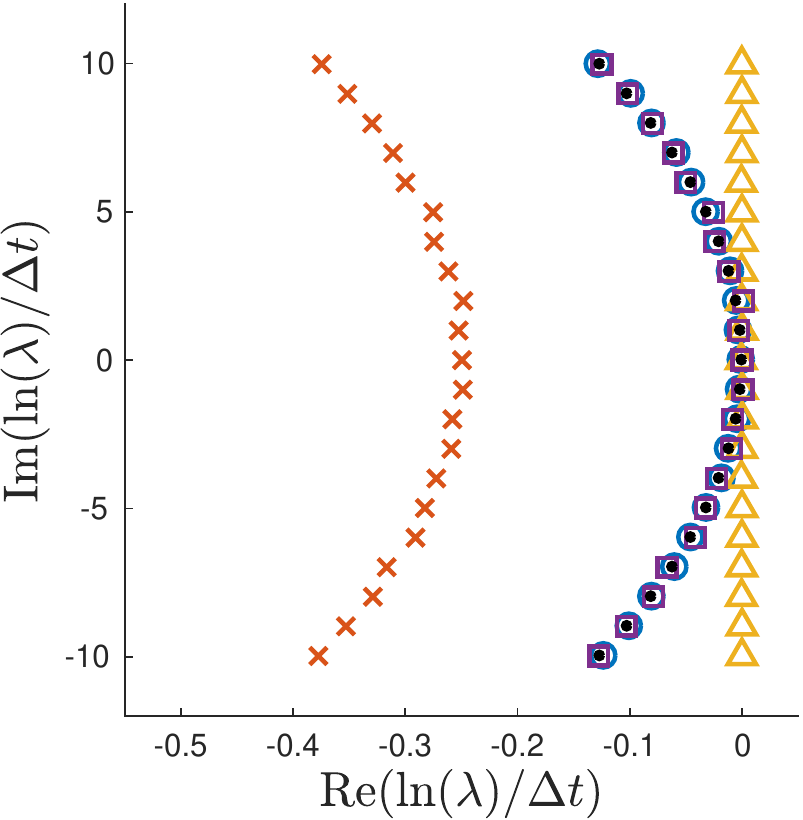}\label{fig:exp_sl_b}}
\caption{(a) Data generated by the noise-free/noisy Stuart--Landau equation and trigonometric observables, which show the phase diffusion when $\sigma_p>0$ (the solid red line). (b--c) The estimated continuous-time eigenvalues, with (b) $\sigma_p=0$ and (c) $\sigma_p=0.5$. Subspace DMD eliminates the effects of the observation noise, keeping the effects of the process noise.}
\label{fig:exp_sl}
\end{figure}
\begin{figure}[t]
\centering
\subfloat[]{\includegraphics[clip,width=0.235\textwidth]{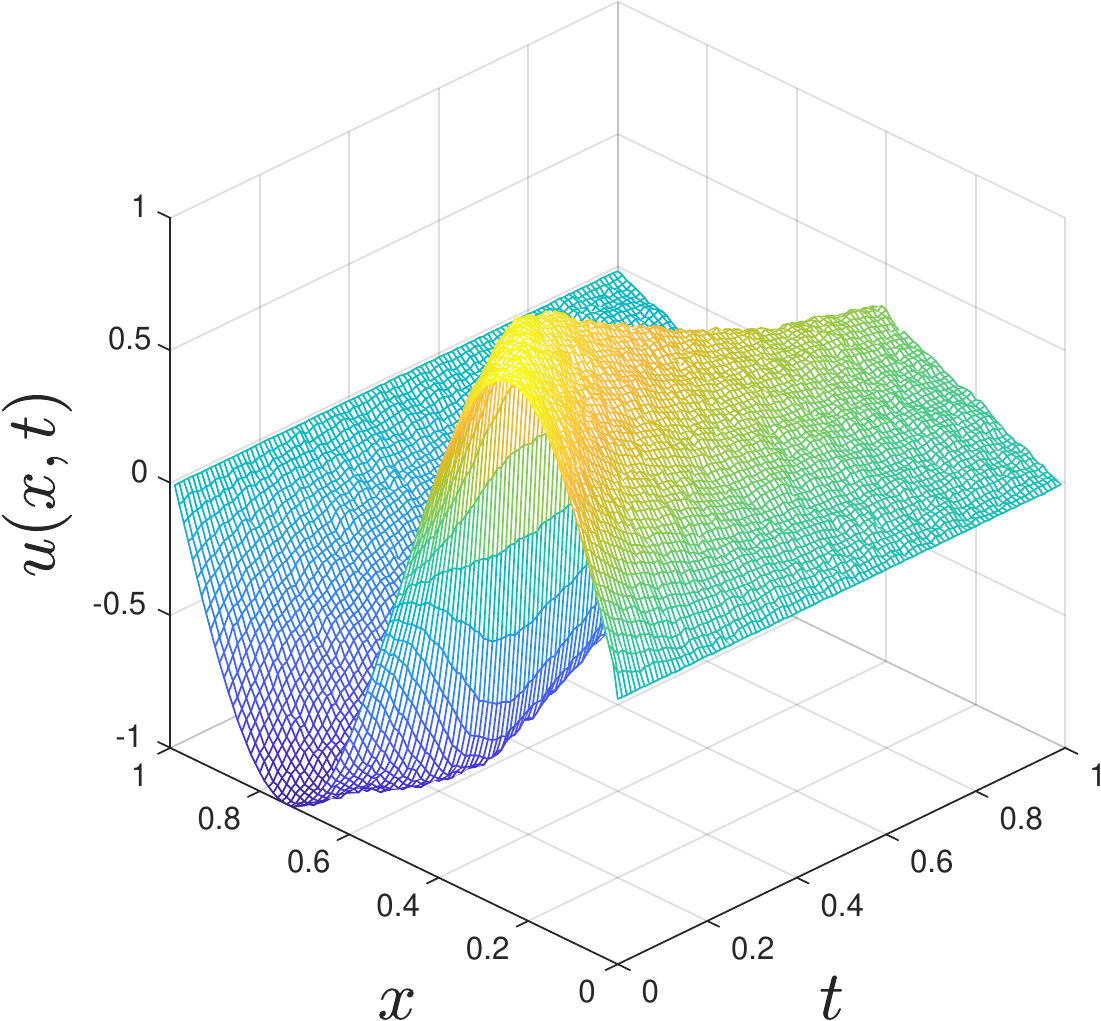}\label{fig:exp_sburgers_data}}
\hfill
\subfloat[]{\includegraphics[clip,width=0.235\textwidth]{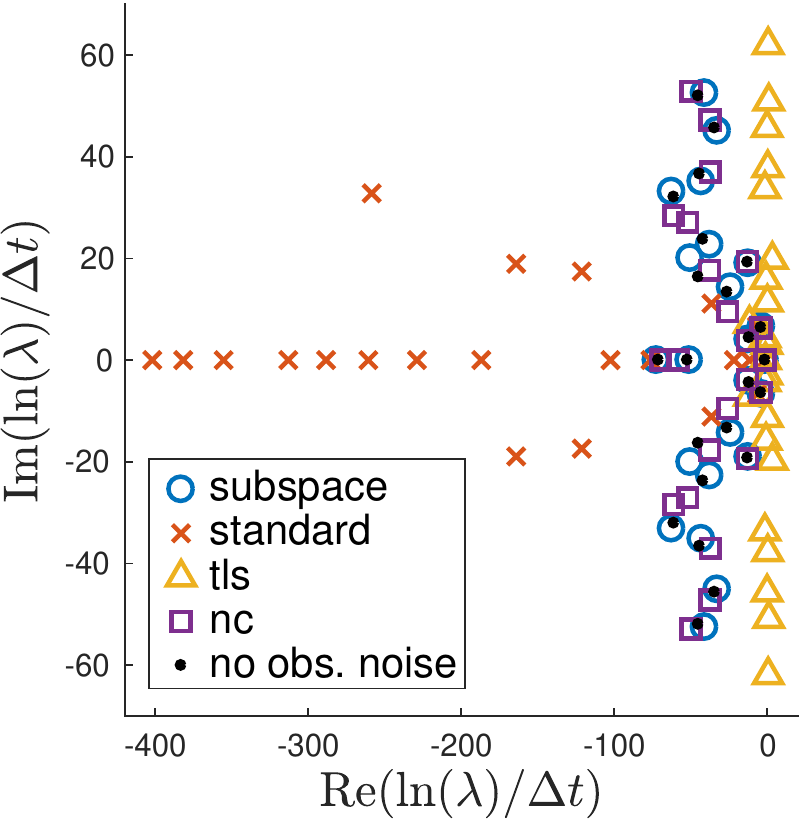}\label{fig:exp_sburgers_result}}
\caption{(a) Data generated by the stochastic Burger's equation with observation noise and (b) the estimated continuous-time eigenvalues. The eigenvalues estimated by subspace DMD agree well with the ones computed with the clean data.}
\label{fig:exp_sburgers}
\end{figure}
%

\subsection{Oscillation perturbed by noise}

The stochastic Stuart--Landau equation on a complex-valued function $z(t)=r(t) \exp(i\theta(t))$ is defined as
\begin{equation*}
\frac{\mathrm{d}z}{\mathrm{d}t} = (\mu+i\gamma)z - (1+i\beta)\vert z \vert^2 z + \sigma_p e(t),
\end{equation*}
where $e(t)$ is Gaussian white noise with unit variance, and $i$ denotes the imaginary unit.
The solution of this equation evolves on a limit cycle at $\vert z \vert=\sqrt{\mu}$ in the absence of process noise (i.e., $\sigma_p=0$).
Bagheri~\cite{Bagheri14} has analyzed the effects of weak process noise on the Koopman eigenvalues of the limit cycle of the Stuart--Landau equation, which can be summarized as follows; the continuous-time eigenvalues lie on the imaginary axis if process noise is absent because the data are completely periodic, but in contrast, when perturbation (phase diffusion, as shown in  Figure~\ref{fig:exp_sl_data}) is present owing to the process noise, a line of the eigenvalues is ``bent''  as shown in Figure~\ref{fig:exp_sl_b}.
Hence, by investigating the distribution of the eigenvalues, one can anticipate the presence and magnitude of the phase diffusion. To this end, we must eliminate the effects of observation noise if any, which produces an extra bias on the eigenvalues, leaving the effects of process noise.

Following the scheme in \cite{Dawson16}, we generated data using the following discretized Stuart--Landau equation in polar-coordinates with process noise:
\begin{equation*}
\begin{bmatrix} r_{t+1} \\ \theta_{t+1} \end{bmatrix} = \begin{bmatrix} r_t + \left( \mu r_t - r_t^3 \right) \Delta t \\ \theta_t + \left( \gamma - \beta r_t^2 \right) \Delta t \end{bmatrix} + \begin{bmatrix} \Delta t & 0 \\ 0 & \Delta t / r_t \end{bmatrix} \bm{e}_t,
\end{equation*}
and a set of noisy trigonometric observables:
\begin{equation*}
\bm{y}_t = \begin{bmatrix} e^{-10i\theta_t} & e^{-9i\theta_t} & \cdots & e^{9i\theta_t} & e^{10i\theta_t} \end{bmatrix} + \bm{w}_t,
\end{equation*}
where the magnitude of the observation noise was fixed to $\sigma_o=0.05$.
We estimated the continuous-time eigenvalues using subspace DMD (Algorithm~\ref{alg:sdmd}), standard DMD (Algorithm~\ref{alg:dmd}), total-least-squares DMD (tls-DMD) \cite{Dawson16,Hemati17}, and noise-corrected DMD (nc-DMD) \cite{Dawson16}.

Figure~\ref{fig:exp_sl_a} shows the eigenvalues without any process noise (i.e., $\sigma_p=0$), and Figure~\ref{fig:exp_sl_b} shows the ones with process noise of $\sigma_p=0.5$. In both plots, we also show the ``clean'' eigenvalues computed with the data without the observation noise.
When the process noise is present (in Figure~\ref{fig:exp_sl_b}), while the eigenvalues estimated by tls-DMD differ from the clean ones (as reported in \cite{Dawson16}), subspace DMD successfully estimates them. Note that the estimation by nc-DMD also coincides with the clean eigenvalues, but nc-DMD needs a precise estimation of magnitude of observation noise, which is often difficult to obtain. Subspace DMD can eliminate the effects of observation noise without such information while {\em keeping the effects of the process noise}.

\begin{figure}[t]
\centering
\subfloat[]{\includegraphics[clip,width=0.235\textwidth]{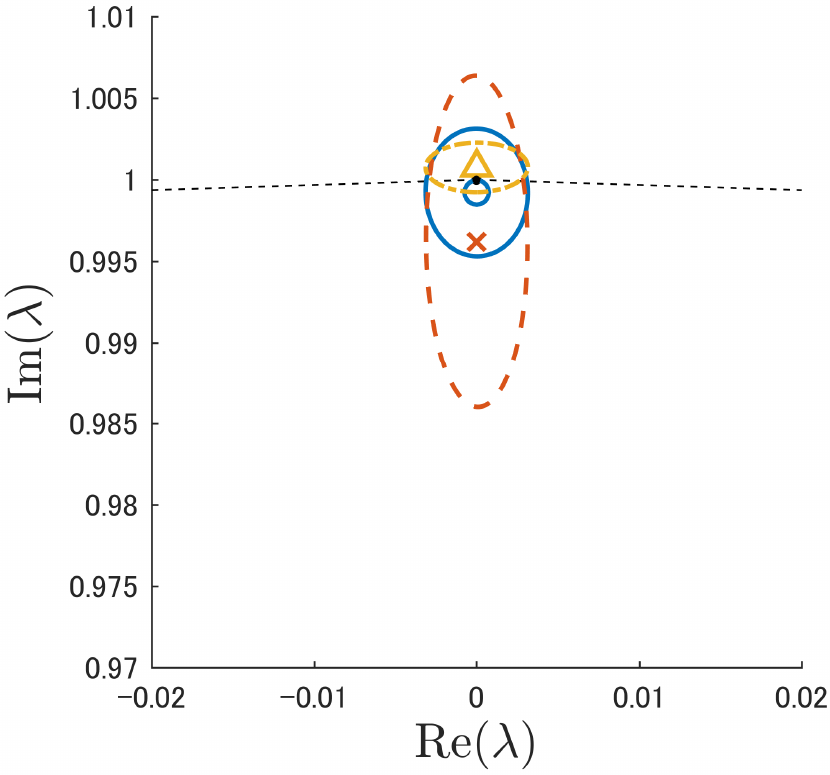}}
\hfill
\subfloat[]{\includegraphics[clip,width=0.233\textwidth]{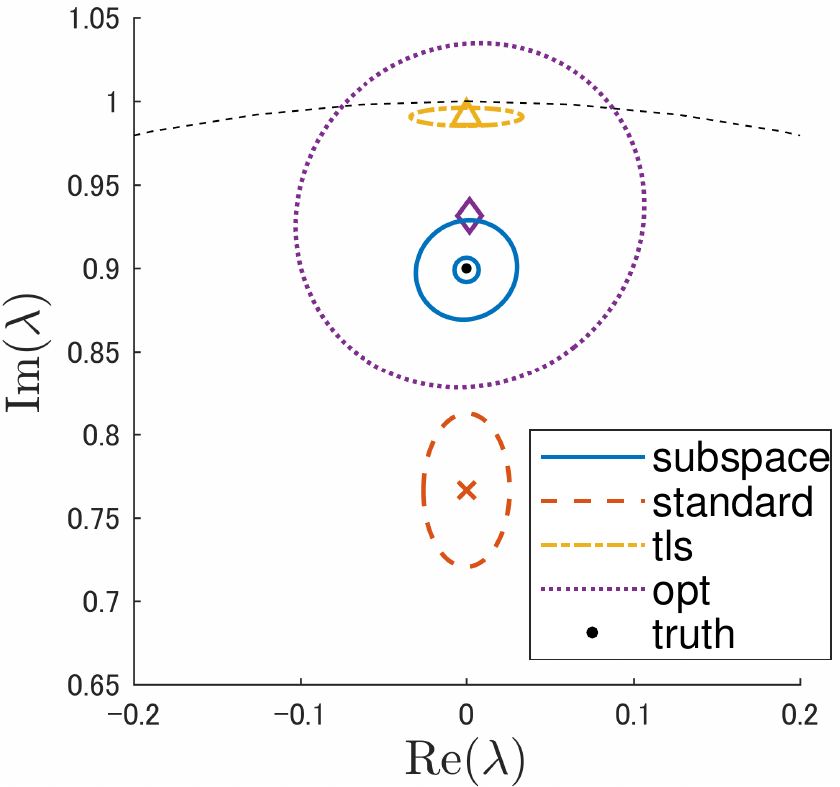}}
\caption{The 95\% confidence intervals and averages of the eigenvalues estimated by subspace DMD, standard DMD, tls-DMD, and opt-DMD on the linear systems, (a) $r=1.0$ and (b) $r=0.9$, with process and observation noises for 1,000 random trials. When $r=0.9$, only subspace DMD shows consistent results.}
\label{fig:exp_linear_1}
\end{figure}
\begin{figure*}[t]
\centering
\begin{minipage}{0.48\textwidth}
\centering
\subfloat[]{\includegraphics[clip,width=0.49\textwidth]{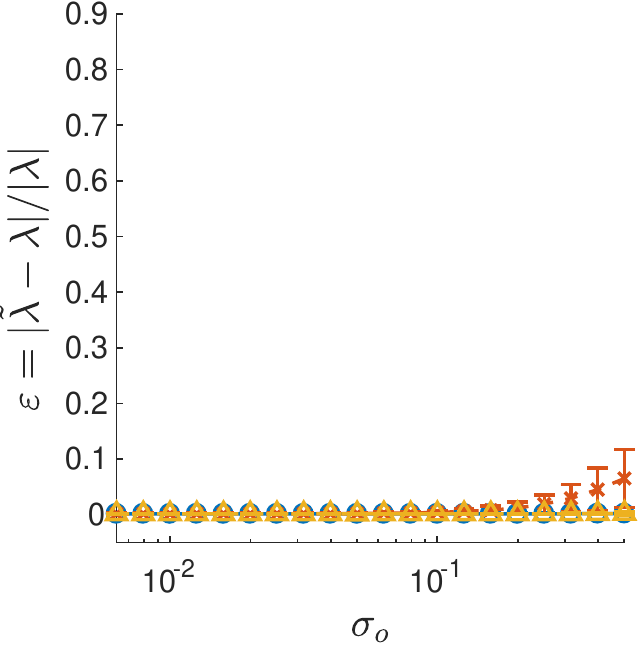}}
\hfill
\subfloat[]{\includegraphics[clip,width=0.49\textwidth]{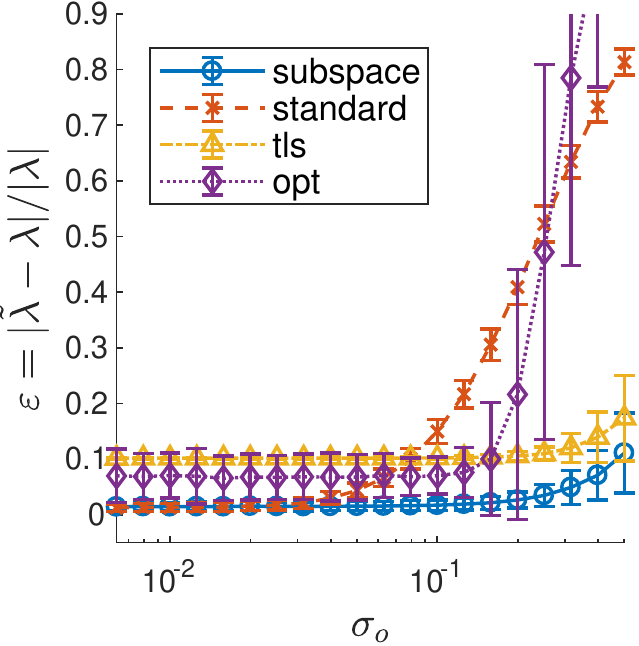}}
\caption{Relative errors of the eigenvalues estimated by subspace DMD, standard DMD, tls-DMD, and opt-DMD on the linear systems, (a) $r=1.0$ and (b) $r=0.9$, with process and observation noises against different magnitudes of observation noise $\sigma_o$. When $r=0.9$, subspace DMD produces much smaller errors compared to the other two methods.}
\label{fig:exp_linear_2}
\end{minipage}
\hspace{0.02\textwidth}
\begin{minipage}{0.48\textwidth}
\centering
\subfloat[$r=1.0$]{\includegraphics[clip,width=0.49\textwidth]{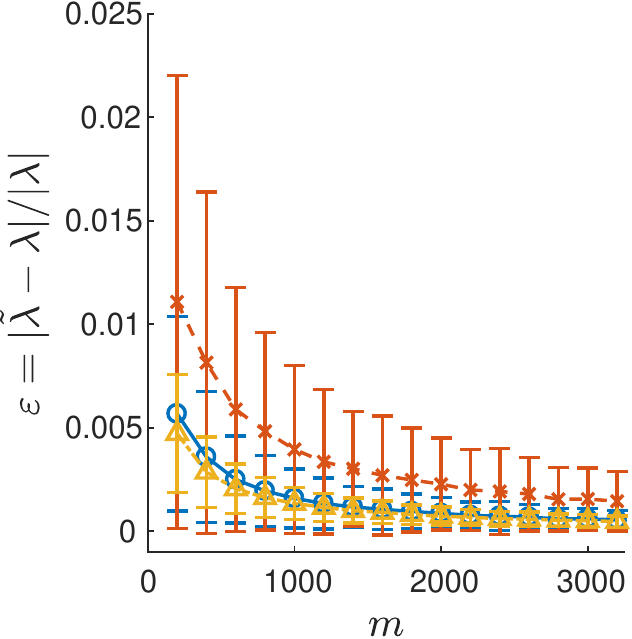}}
\hfill
\subfloat[$r=0.9$]{\includegraphics[clip,width=0.49\textwidth]{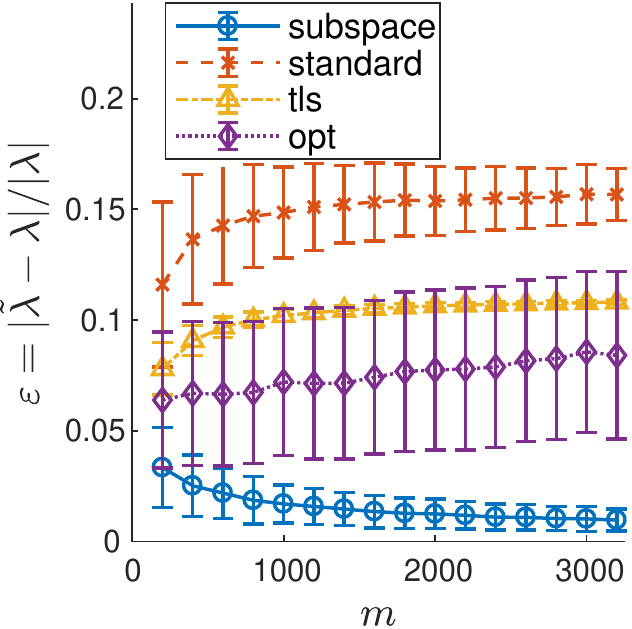}}
\caption{Relative errors of the eigenvalues estimated by subspace DMD, standard DMD, tls-DMD, and opt-DMD on the linear systems, (a) $r=1.0$ and (b) $r=0.9$, with process and observation noises against different numbers of snapshots $m$ fed into the algorithms. When $r=0.9$, only the output of subspace DMD converges to the true value.}
\label{fig:exp_linear_3}
\end{minipage}
\end{figure*}
%

\subsection{Noisy damping modes}

Let us consider the stochastic Burger's equation
\begin{equation*}
\partial_tu(x,t) + u\partial_x u = k\partial^2_xu + \sigma_p e(x,t)
\end{equation*}
with $k>0$ and Gaussian space-time white noise $e(x,t)$.
In fact, the eigenvalues of the Koopman operator on Burger's equation (without process noise, i.e., $\sigma_p=0$) can be analytically obtained via the Cole--Hopf transformation and they correspond to the decaying modes of the solution \cite{Budisic12,Kutz16b}.
When process noise is present ($\sigma_p>0$), the solution of Burger's equation becomes ``rough,'' but its global appearance remains similar to the case of no process noise, as shown in Figure~\ref{fig:exp_sburgers_data}.

We approximated the solution of the stochastic Burger's equation with $k=0.01$ and $\sigma_p=0.01$ using Crank--Nicolson--Maruyama method (see, e.g., \cite{Hausenblas03}) with initial condition $u(x,0)=\sin(2 \pi x)$ and Dirichlet boundary condition $u(0,t)=u(1,t)=0$, setting the ranges by $x \in \left[ 0,1 \right]$ and $t \in \left[ 0,1 \right]$ and the discretization step sizes by $\Delta x=1 \times 10^{-2}$ and $\Delta t = 5 \times 10^{-5}$.
Based on the approximated solution, we finally generated data with observation noise
\begin{equation*}
\bm{y}_t = \begin{bmatrix} u(0,t) & u(\Delta x,t) & u(2\Delta x,t) & \dots & u(1,t) \end{bmatrix}^\tr + \bm{w}_t,
\end{equation*}
where the magnitude of the observation noise was set $\sigma_o=0.001$.
The estimated eigenvalues are plotted in Figure~\ref{fig:exp_sburgers_result}. While the eigenvalues obtained by tls-DMD lie approximately on the imaginary axis, the estimation by subspace DMD agrees well with the eigenvalues computed with data that contain no observation noise. Again note that, though the estimation by nc-DMD also aligns with the clean eigenvalues, it requires a precise estimation of observation noise magnitude.

\subsection{Quantitative investigation of effects of noises}

Let us investigate the performance of subspace DMD quantitatively using a simple linear system. We generated data using a linear time-invariant system
\begin{equation*}
\bm{x}_t = \begin{bmatrix} \lambda & 0 \\ 0 & \bar{\lambda} \end{bmatrix}\bm{x}_{t-1} + \bm{e}_t,\quad
\lambda=ri,
\end{equation*}
whose Koopman eigenvalues obviously contain $\lambda$ and $\bar{\lambda}$. Moreover, we used the identity observable with observation noise
\begin{equation*}
\bm{y}_t = \bm{x}_t + \bm{w}_t.
\end{equation*}
We fixed the standard deviation of the process noise to $\sigma_p=0.1$, the eigenvalue to $\lambda=ri$ with $r=1.0$ or $0.9$, and the initial state to $\bm{x}_0=\begin{bmatrix}1&1\end{bmatrix}^\tr$. Hence, this system exhibits oscillation perpetuated by the process noise when $r=1.0$ and is damped while being excited by the process noise when $r=0.9$.
We applied subspace DMD, standard DMD, tls-DMD, and optimized DMD (opt-DMD) \cite{Chen12} to multiple datasets generated with different random seeds. In those experiments, we have found that opt-DMD is unstable when $r=1.0$ and it does not output much reasonable results because it needs to compute exponentials of eigenvalues. Hence, the results of opt-DMD when $r=1.0$ are not plotted in Figures~\ref{fig:exp_linear_1}, \ref{fig:exp_linear_2}, and \ref{fig:exp_linear_3}.

In Figure~\ref{fig:exp_linear_1}, we show the 95\% confidence intervals of the estimated eigenvalues for 1,000 random trials with observation noise of magnitude $\sigma_o=0.1$ and $m=$~1,000 snapshots. When $r=1.0$, the estimations by subspace DMD, standard DMD, and tls-DMD scatter around the true value, while the results of standard DMD deviate a little more than the others. When $r=0.9$, the estimations by standard DMD, tls-DMD, and opt-DMD deviate from the true value; only the outputs of subspace DMD distribute around the true value.

In Figure~\ref{fig:exp_linear_2}, the relative errors $\varepsilon = \vert \tilde{\lambda} - \lambda \vert / \vert \lambda \vert$ of estimated eigenvalues $\tilde{\lambda}$ are plotted against different magnitudes of the observation noise $\sigma_o$, with the number of snapshots fed into the algorithms being fixed by $m=$~1,000. When $r=1.0$, subspace DMD, standard DMD, and tls-DMD work almost equally well. When $r=0.9$, while the errors of standard DMD and opt-DMD rapidly grow and tls-DMD generates a regular bias, subspace DMD produces almost no bias and is tolerant to the observation noise.

In Figure~\ref{fig:exp_linear_3}, relative errors $\varepsilon$ are plotted against different $m$ with fixed $\sigma_o=0.1$. When $r=1.0$, subspace DMD, standard DMD, and tls-DMD converge when $m$ becomes large. When $r=0.9$, only subspace DMD converges, which is expected from Theorem~\ref{thm:subspace}.

\subsection{Low-rank high-dimensional data}

\begin{figure}[t]
\centering
\subfloat[]{\includegraphics[clip,width=0.23\textwidth]{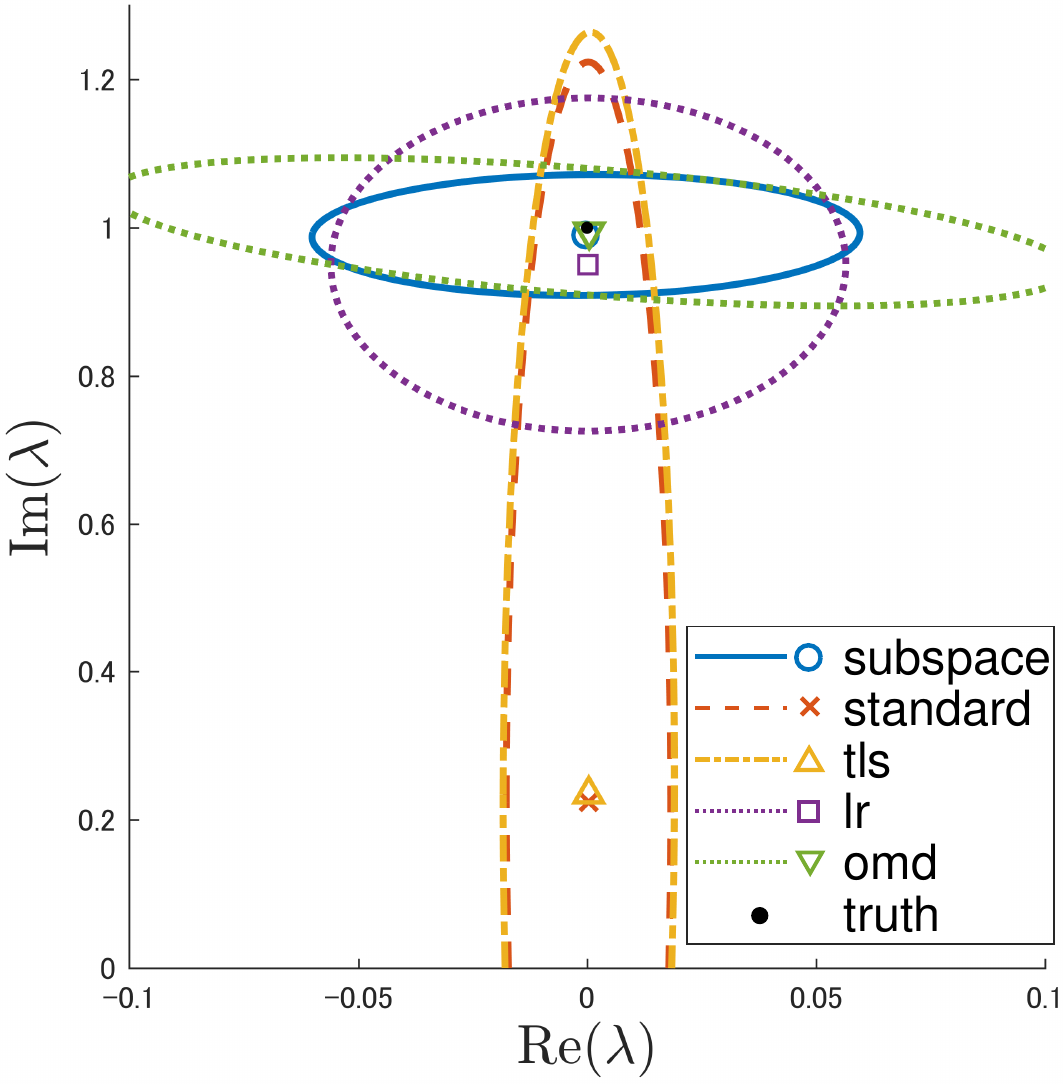}\label{fig:exp_highdimlinear_a}}
\hfill
\subfloat[]{\includegraphics[clip,width=0.23\textwidth]{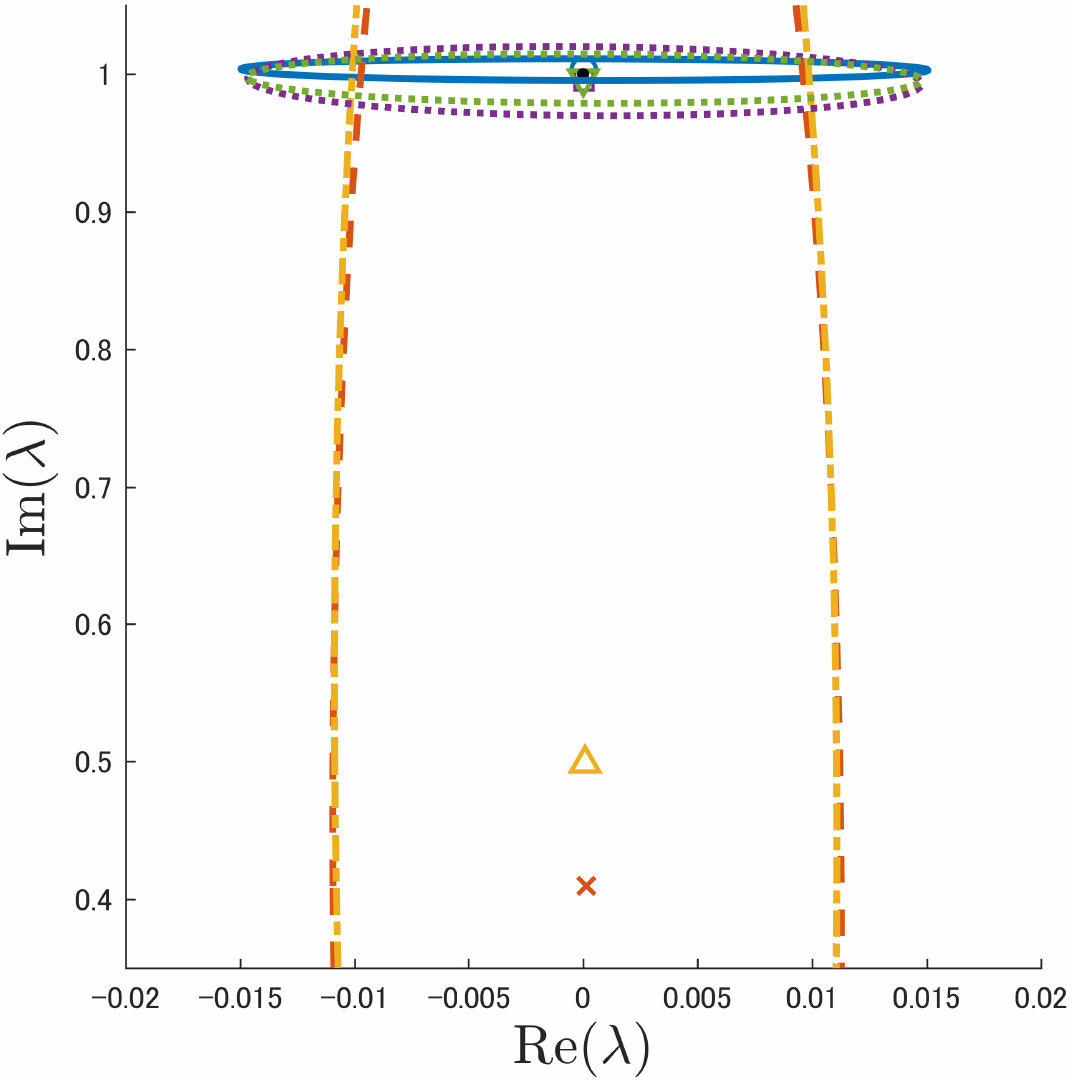}\label{fig:exp_highdimlinear_b}}
\caption{The $95$\% confidence intervals and averages of the eigenvalues estimated by subspace DMD, standard DMD, tls-DMD, lr-DMD, and OMD on the high-dimensional ($n=500$) low-rank ($r=2$) system for 1,000 random trials. The sample sizes are (a) $m=50$ and (b) $m=200$. In both cases, the outputs of subspace DMD distribute around the true values denoted by the black filled circle.}
\label{fig:exp_highdimlinear}
\end{figure}

We have shown the convergence of subspace DMD in the large sample limit in Theorem~\ref{thm:subspace}, but in practice, DMD is often applied in a high-dimensional setting, where the number of snapshots $m$ is much less than dimensionality of data $n$. 
To simulate such circumstances, we generated 500-dimensional data using a linear time-invariant system:
\begin{equation*}
\begin{aligned}
\bm{x}_t &= \bm{L} \begin{bmatrix}i&0\\0&-i\end{bmatrix} \bm{L}^\tr \bm{x}_{t-1} + \bm{e}_t,\\
\bm{y}_t &= \bm{x}_t + \bm{w}_t,
\end{aligned}
\end{equation*}
where $\bm{L} \in \mathbb{R}^{500 \times 2}$ satisfies $\bm{L}^\tr\bm{L}=\bm{I}$, with $\sigma_p=0.1$ and $\sigma_o=0.1$. We prepared two datasets with different sizes, $m=50$ and $m=200$, and applied the following DMD variants: subspace DMD, standard DMD, tls-DMD, optimal low-rank DMD (lr-DMD) \cite{Heas17}, and optimal mode decomposition (OMD) \cite{Wynn13}. For each method, we introduced the way to obtain a low-rank solution; only the first two POD modes were used in the standard DMD and tls-DMD, the rank parameter was set to two in lr-DMD and OMD, and only the first two columns of $\bm{U}_q$ were used in subspace DMD.

In Figure~\ref{fig:exp_highdimlinear_a}, we show the $95$\% confidence intervals and averages of the estimated eigenvalues for 1,000 random trials with $m=50$. While the estimations by standard DMD and tls-DMD deviate far from the true value because of the process noise (and observation noise), the estimations by subspace DMD, lr-DMD, and OMD distribute around the true value. In particular, the variance of the estimation by subspace DMD is smaller than that of lr-DMD and OMD. Figure~\ref{fig:exp_highdimlinear_b} shows the results with $m=200$; in this case, the distributions of the estimations by subspace DMD, lr-DMD, and OMD almost coincide.

\subsection{Application: cylinder wake}

\begin{figure}[t]
\centering
\subfloat[]{\includegraphics[clip,width=0.235\textwidth,valign=t]{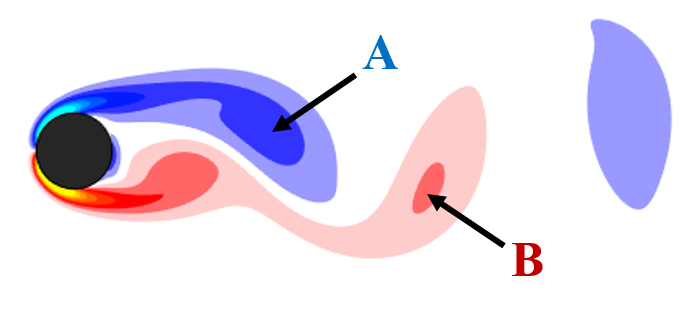}}
\hfill
\subfloat[]{\includegraphics[clip,width=0.225\textwidth,valign=t]{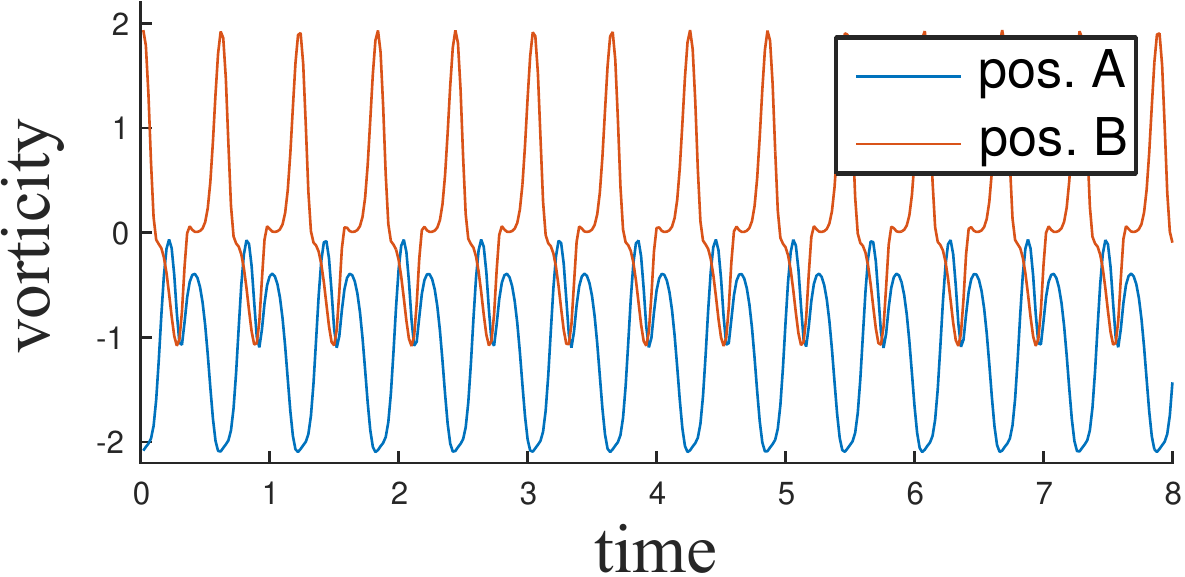}}
\caption{(a) Example snapshot of the vorticity field in the limit-cycle characterized by the K{\'a}rm{\'a}n vortex street. (b) Time variation of the vorticity at locations A and B.}
\label{fig:exp_cylinder_data}
\end{figure}

As an example of an application, we applied subspace DMD to a two-dimensional flow past a circular cylinder with Reynolds number $\mathrm{Re}=100$. We generated data using a solver based on the fast immersed boundary method with the multi-domain technique \cite{Taira07,Colonius08} with four nested grids, each of which contains $450 \times 200$ points. The diameter of the cylinder corresponds to $50$ points in the finest grid. The solver uses the third-order Runge--Kutta method with time-step $\Delta t = 0.02$. We collected $400$ snapshots of the vorticity fields with intervals of size $10\Delta t$ from the limit cycle characterized by the K{\'a}rm{\'a}n vortex street. An example of the snapshots (without observation noise) and the time-variation of vorticity at two locations (A and B) are shown in Figure~\ref{fig:exp_cylinder_data}.
We applied subspace DMD, standard DMD, and tls-DMD to the data contaminated with observation noise of $\sigma_o=0.1$. Every method was run with a low-rank approximation of $r=15$ because the first $15$ POD modes contained about $99.9$\% of the energy of the original data.

In Figure~\ref{fig:exp_cylinder_do1_eigs}, the eigenvalues estimated with the noisy dataset and the noise-free dataset are plotted; subspace DMD and tls-DMD generate smaller biases than standard DMD does. The eigenvalues are numbered from one to seven in Figure~\ref{fig:exp_cylinder_do1_eigs}, according to their frequency (the magnitude of the imaginary part).
In the remainder of Figure~\ref{fig:exp_cylinder_do1}, we show the dynamic modes corresponding to eigenvalue 1 ($\sim 10i$) in the upper row and eigenvalue 4 ($\sim 40i$) in the lower row. We confirm that no adversarial effect is present in the dynamic modes computed by subspace DMD.
Note that, in this cylinder wake experiment, no process noise (except for small errors due to the numerical integration) is involved.
Subspace DMD is also applicable to classical (but frequent) situations of data analysis like this, where almost no process noise is present.

\begin{figure*}[t]
\centering
\begin{minipage}{0.2\textwidth}
\centering
\subfloat[]{\includegraphics[clip,width=0.98\textwidth]{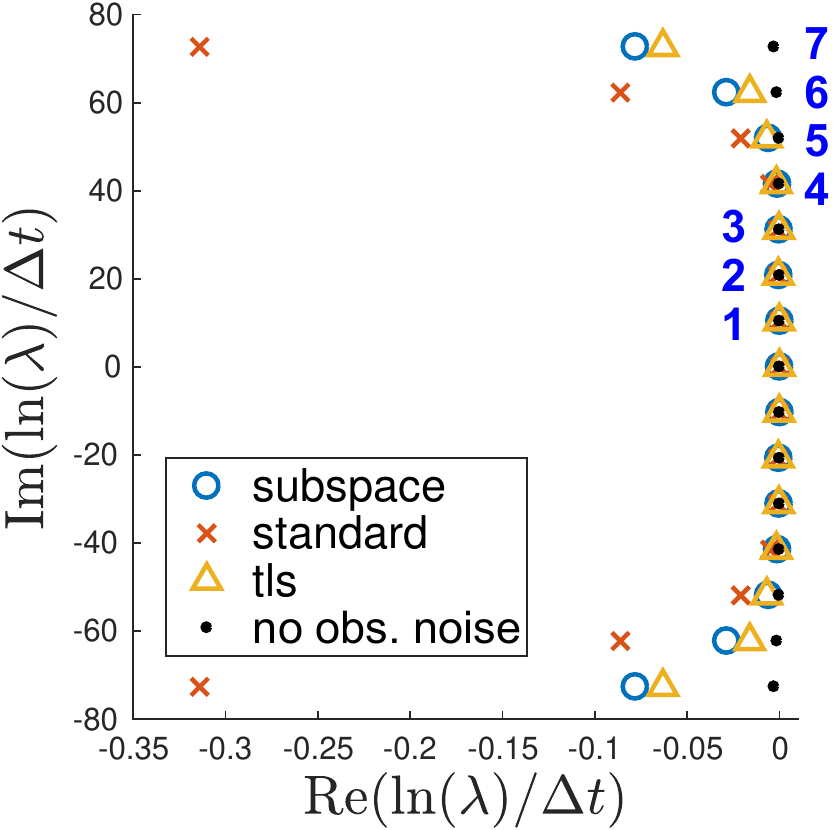}\label{fig:exp_cylinder_do1_eigs}}
\end{minipage}
\hspace{0.02\textwidth}
\begin{minipage}{0.75\textwidth}
\centering
\begin{tabular}{cccc}
\subfloat{\includegraphics[clip,width=0.24\textwidth]{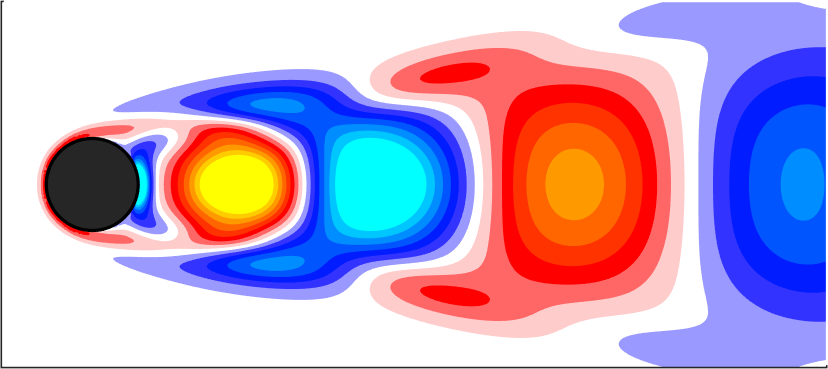}}
&
\subfloat{\includegraphics[clip,width=0.24\textwidth]{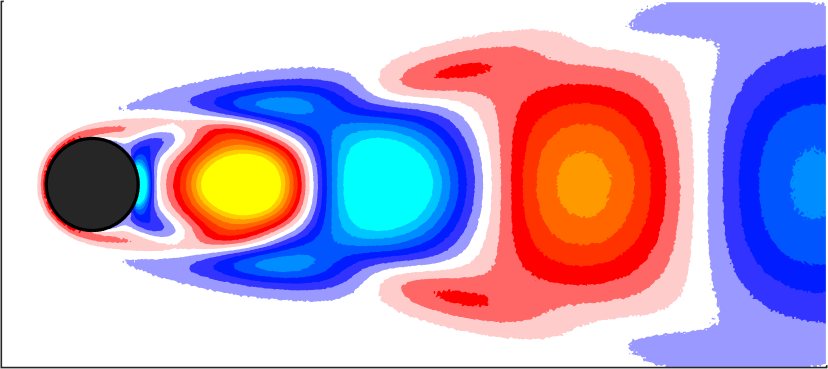}}
&
\subfloat{\includegraphics[clip,width=0.24\textwidth]{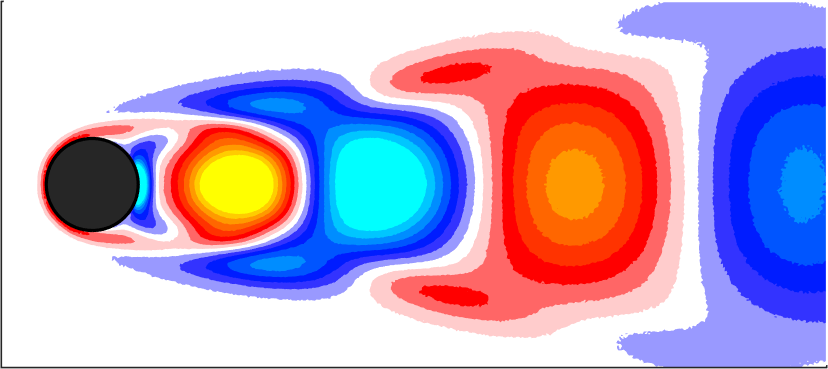}}
&
\subfloat{\includegraphics[clip,width=0.24\textwidth]{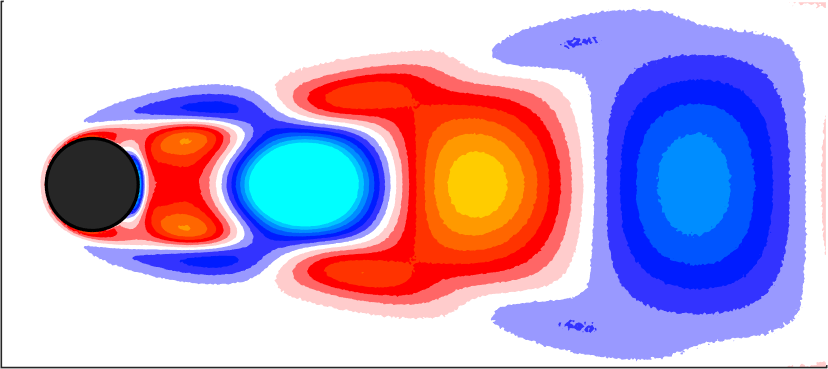}}
\\
\subfloat[]{\addtocounter{subfigure}{-4}\includegraphics[clip,width=0.24\textwidth]{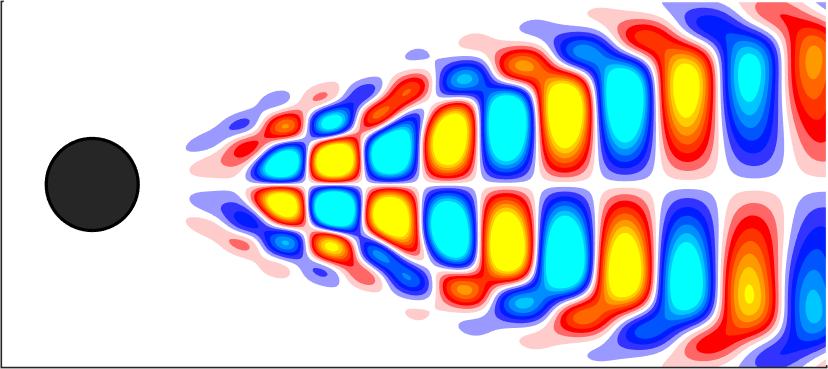}}
&
\subfloat[]{\includegraphics[clip,width=0.24\textwidth]{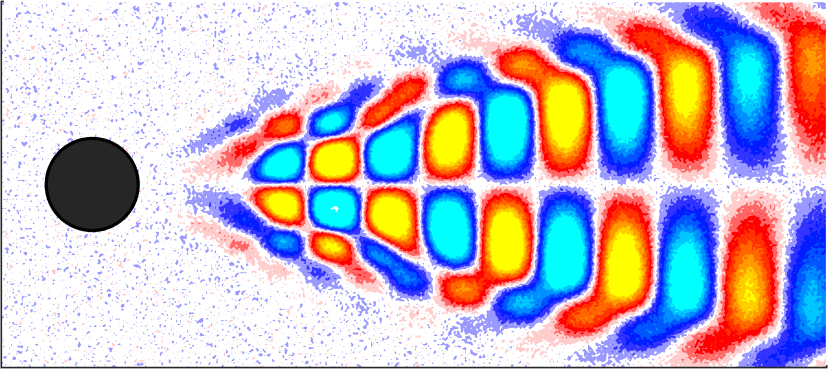}}
&
\subfloat[]{\includegraphics[clip,width=0.24\textwidth]{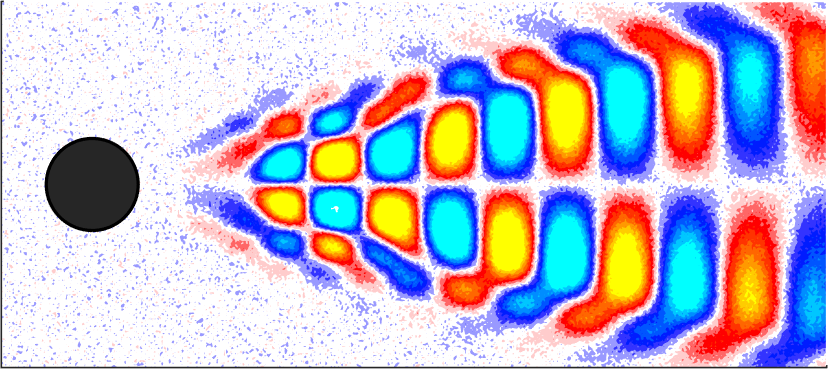}}
&
\subfloat[]{\includegraphics[clip,width=0.24\textwidth]{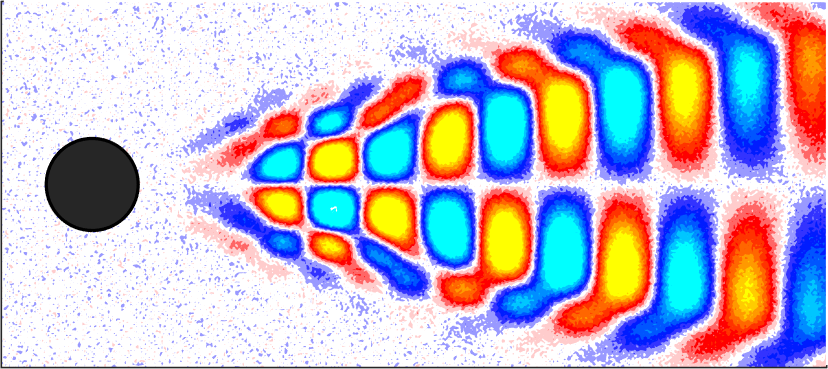}}
\end{tabular}
\end{minipage}
\caption{(a) Continuous-time eigenvalues estimated on the noise-free and noisy datasets, which are numbered from one to seven according to their frequency (i.e., imaginary part). (b) Dynamic modes computed on the noise-free dataset by standard DMD. (c--e) Dynamic modes computed on the noisy dataset by (c) subspace DMD, (d) standard DMD, and (e) tls-DMD. The upper row corresponds to eigenvalue 1 ($\sim 10i$) and the lower row corresponds to eigenvalue 4 ($\sim 40i$). No adversarial effects are present in the results of subspace DMD, even without the process noise. In (b--e), the magnitude of the dynamic modes are normalized to a common color scheme. Best viewed in color.}
\label{fig:exp_cylinder_do1}
\end{figure*}

\section{Discussion}\label{sec:discussion}

This section contains additional discussion on related methods and important properties of Koopman spectral analysis and DMD algorithms that have not been addressed sufficiently in this paper. These issues may imply future directions of research.

\subsection{Relation to subspace system identification and extension to controlled systems}

Subspace DMD has a strong connection to the methods called {\em subspace system identification} (see, e.g., \cite{VanOverschee96,Katayama05}) in their computational methodology. Subspace system identification is a series of methods mainly for the identification of linear time-invariant systems, whereas in this paper, we present a similar methodology for nonlinear dynamical systems involving the observables and the stochastic Koopman operator.

Subspace system identification has been studied from the viewpoint of control theory and admits the presence of {\em input signals} distinguished from the process noise. Therefore, an extension of subspace DMD to controlled dynamical systems would be straightforward following the methodologies developed in the research of subspace system identification. Also, one may take a closed-loop controlled system into consideration. In the context of DMD, Proctor~\etal~\cite{Proctor16a} have discussed a variant of DMD for data obtained from the controlled systems.

\subsection{Construction of Koopman invariant subspace}
\label{subsec:invariant}

The key point of DMD as a numerical realization of the Koopman analysis lies in preparing a set of observables that spans a subspace invariant to $\mathcal{K}$. Several researchers have worked on this issue; Williams~\etal~\cite{Williams15a} proposed using a user-defined dictionary of observables to adopt DMD to highly nonlinear systems, and Brunton~\etal \cite{Brunton16a} utilized an identification technique based on a sparse regression \cite{Brunton16c} to identify the dynamic-specific observables to be used. In addition, Kawahara \cite{Kawahara16} defined the Koopman analysis for observables in reproducing kernel Hilbert spaces to build a theory of DMD based on the reproducing kernels.

Another option, especially for deterministic systems, is to use delay coordinates, i.e., stacking observations at neighboring timestamps in each column of the data matrices. In general, a Krylov-like sequence of observables $\{g,\mathcal{K}g,\mathcal{K}^2g,\dots\}$ rapidly becomes almost linearly dependent, and thus can be used to obtain a subspace that is {\em approximately} invariant to $\mathcal{K}$. Based on the delayed measurements, we obtain a data matrix as a Hankel matrix. The use of delay coordinates for DMD was first discussed by Tu~\etal~\cite{Tu14}, and Brunton~\etal~\cite{Brunton17} mentioned DMD based on Hankel matrices, referring to the well-known Taken's theorem \cite{Takens81}. Susuki and Mezi{\'c} \cite{Susuki16} defined an approximation of the Koopman analysis using Prony's method, which also uses Hankel matrices. Arbabi and Mezi{\'c} \cite{Arbabi16} {showed} the convergence of DMD on Hankel matrices {build with an observable contained in a Koopman invariant subspace}.
However, {since delay coordinates with a linear monomial cannot span a Koopman invariant subspace of nonlinear systems exactly,} one should use a combination of the nonlinear observables and the delay coordinates.

\section{Conclusion}
\label{sec:conclusion}

In this work, we developed {\em subspace DMD}, an algorithm for stochastic Koopman analysis with noisy observations. We have shown that the output of the proposed algorithm converges to the spectra of the stochastic Koopman operator in the large sample limit even if both process noise and observation noise are present. Moreover, we have shown its empirical performance with the numerical examples on different types of random dynamical systems. We also discussed the possible future directions of research, such as an extension to controlled dynamical systems and development of a way to obtain Koopman invariant subspaces.

\begin{acknowledgments}
This work was supported by JSPS KAKENHI Grant No. {JP15J09172, JP26280086, JP16H01548, and JP26289320}.
\end{acknowledgments}

\bibliographystyle{apsrev4-1}
\bibliography{main}

\appendix*
\section{Proofs of lemmas in Section~\ref{sec:noisy}}

\subsection{Proof of Lemma~\ref{lem:p_g}}

From Eq.~\eqref{eq:g_sol}, for the case of $t'>t$, we have
\begin{widetext}\begin{equation*}
\begin{aligned}
\bm{G}_{{t'},{t}}
&= \mathbb{E}_\Omega \left[ \left( \bm{K}_\Omega^{t'}\bm{g}(\bm{x}_0) + \sum_{k=0}^{{t'}-1}\bm{K}_\Omega^{{t'}-k-1}\bm{e}_k \right) \left( \bm{K}_\Omega^{{t}}\bm{g}(\bm{x}_0) + \sum_{k=0}^{{t}-1}\bm{K}_\Omega^{{t}-k-1}\bm{e}_k \right)^\ct \right]\\
&= \bm{K}_\Omega^{{t'}-{t}} \mathbb{E}_\Omega \left[ \left( \bm{K}_\Omega^{{t}}\bm{g}(\bm{x}_0) + \sum_{k=0}^{{t}-1}\bm{K}_\Omega^{{t}-k-1}\bm{e}_k + \sum_{k={t}}^{{t'}-1}\bm{K}_\Omega^{{t}-k-1}\bm{e}_k \right) \left( \bm{K}_\Omega^{{t}}\bm{g}(\bm{x}_0) + \sum_{k=0}^{{t}-1}\bm{K}_\Omega^{{t}-k-1}\bm{e}_k \right)^\ct \right]\\
&= \bm{K}_\Omega^{{t'}-{t}} \bm{G}_{t,t} + \bm{K}_\Omega^{{t'}-{t}} \mathbb{E}_\Omega \left[ \left( \sum_{k={t}}^{{t'}-1}\bm{K}_\Omega^{{t}-k-1}\bm{e}_k \right) \left( \bm{K}_\Omega^{{t}}\bm{g}(\bm{x}_0) + \sum_{k=0}^{t-1}\bm{K}_\Omega^{{t}-k-1}\bm{e}_k \right)^\ct \right]\\
&= \bm{K}_\Omega^{{t'}-{t}} \bm{G}_{t,t} + \sum_{k=t}^{t'-1} \bm{K}_\Omega^{t'-k-1} \mathbb{E}_\Omega \left[ \bm{e}_k \right] \bm{g}(\bm{x}_0)^\ct \left( \bm{K}_\Omega^t \right)^\ct + \sum_{k=0}^{t-1}\sum_{k'=t}^{t'-1} \bm{K}_\Omega^{t'-k'-1} \mathbb{E}_\Omega \left[ \bm{e}_{k'}\bm{e}_k^\ct \right] \left( \bm{K}_\Omega^{t-k-1} \right)^\ct \\
&= \bm{K}_\Omega^{{t'}-{t}} \bm{G}_{t,t},
\end{aligned}
\end{equation*}\end{widetext}
where the last equality holds because $\bm{e}$ is zero-mean and because of Assumption~\ref{asmp:e_cov}.
For the case of $t'>t$, from the definition of $\bm{G}_{t',t}$ and the above equation, we have
\begin{equation*}
\bm{G}_{t',t} = \bm{G}_{t,t'}^\ct = \bm{G}_{t',t'} \left( \bm{K}_\Omega^{t-t'} \right)^\ct.
\end{equation*}
%

\subsection{Proof of Lemma~\ref{lem:p_h}}

From Eqs.~\eqref{eq:g_sol}~and~\eqref{eq:h_def}, for the case of $t'>t$, we have
\begin{widetext}\begin{equation*}
\begin{aligned}
\bm{H}_{{t'},{t}}
&= \mathbb{E}_{\Omega,S} \left[ \left( \bm{g}(\bm{x}_{t'}) + \bm{w}_{t'} \right) \left( \bm{g}(\bm{x}_t) + \bm{w}_t \right)^\ct \right] \\
&= \mathbb{E}_{\Omega,S} \left[ \bm{g}(\bm{x}_{t'})\bm{g}(\bm{x}_t)^\ct + \bm{g}(\bm{x}_{t'})\bm{w}_t^\ct + \bm{w}_{t'}\bm{g}(\bm{x}_t)^\ct + \bm{w}_{t'}\bm{w}_t^\ct \right] \\
&= \bm{K}_\Omega^{t'-t}\bm{G}_{t,t} +  \sum_{k=0}^{t'-1} \bm{K}_\Omega^{t'-k-1} \mathbb{E}_{\Omega,S} \left[\bm{e}_k\bm{w}_t^\ct \right] + \left( \sum_{k=0}^{t-1} \bm{K}_\Omega^{t-k-1} \mathbb{E}_{\Omega,S} \left[ \bm{e}_k \bm{w}_{t'}^\ct \right] \right)^\ct  \\
&= \bm{K}_\Omega^{t'-t}\bm{G}_{t,t} +  \bm{K}_\Omega^{t'-t-1} \bm{R} \\
&= \bm{K}_\Omega^{t'-t-1} \left( \bm{K}_\Omega \bm{G}_{t,t}  + \bm{R} \right),
\end{aligned}
\end{equation*}\end{widetext}
where the third and the fourth equalities are from Assumption~\ref{asmp:w_cov}. When $t'=t$, from Assumption~\ref{asmp:w_cov}, we have
\begin{equation*}
\begin{aligned}
\bm{H}_{t,t}
&= \mathbb{E}_{\Omega,S} \left[ \left( \bm{g}(\bm{x}_{t}) + \bm{w}_{t} \right) \left( \bm{g}(\bm{x}_t) + \bm{w}_t \right)^\ct \right] \\
&= \bm{G}_{t,t} + \bm{Q}.
\end{aligned}
\end{equation*}

\end{document}